\documentclass[a4paper,12pt,reqno]{amsart}
\usepackage[latin1]{inputenc}
\usepackage[T1]{fontenc}
\usepackage{amsmath,amssymb}
\usepackage{amsthm}
\usepackage{mathrsfs}
\usepackage{mathtools} 
\usepackage{fancyhdr}
\usepackage{graphicx}
\usepackage{esint,color,xcolor}
\usepackage{verbatim}
\usepackage[english]{babel}
\usepackage[a4paper,top=3.28cm,bottom=3.2cm,left=2.8cm,right=2.8cm]{geometry}
\usepackage{yfonts}
\usepackage{hyperref}

\newcommand{\B}{\mathbb{B}}

\newcommand{\N}{\mathbb{N}}

\newcommand{\R}{\mathbb{R}}

\newcommand{\cA}{\mathcal{A}}

\newcommand{\cD}{\mathcal{D}}

\newcommand{\cF}{\mathcal{F}}

\newcommand{\cH}{\mathcal{H}}

\newcommand{\cS}{\mathcal{S}}

\newcommand{\cs}{\mathfrak s}

\newcommand{\bX}{{\bf X}}
\newcommand{\bY}{{\bf Y}}

\newcommand{\bU}{{\bf U}}

\newcommand{\ep}{\varepsilon}

\newcommand{\ph}{\varphi}

\newcommand{\sm}{\setminus}

\newcommand{\diam}{\mbox{\rm diam}}

\renewcommand{\exp}{\mbox{\rm exp}\;\!}

	\renewcommand{\div}{\mbox{\rm div}}

	\newcommand{\Lie}{\mathrm{Lie}}
	\newcommand{\rE}{{\rm E}}
	
	\newcommand{\m}{\mathrm m}
	\newcommand{\n}{\mathrm n}

	\newcommand{\res}{\mathop{\hbox{\vrule height 7pt width .5pt depth 0pt \vrule height .5pt width 6pt depth 0pt}}\nolimits}

	\newcommand{\beqas}{\begin{eqnarray*}}
		\newcommand{\eeqas}{\end{eqnarray*}}
	\newcommand{\beqa}{\begin{eqnarray}}
		\newcommand{\eeqa}{\end{eqnarray}}
	\newcommand{\beq}{\begin{equation}}
		\newcommand{\eeq}{\end{equation}}
	\newcommand{\bce}{\begin{center}}
		\newcommand{\ece}{\end{center}}
	
	\newcommand{\pa}[1]{\left( #1 \right)}               
	\newcommand{\set}[1]{\left\{ #1 \right\}}            

\theoremstyle{definition}
\newtheorem{definition}{Definition}[section]
\newtheorem{rmk}[definition]{Remark}
\newtheorem{notation}[definition]{Notation}

\theoremstyle{plain}
\newtheorem{lemma}[definition]{Lemma}
\newtheorem{proposition}[definition]{Proposition}
\newtheorem{corollary}[definition]{Corollary}
\newtheorem{theorem}[definition]{Theorem}

\newcommand{\dive}{{\mathrm{div}}}
\newcommand{\der}{\mathrm{d}}
\newcommand{\dd}{\mathrm{d}}

\newcommand{\bu}{{\bf 1}}

\title[Surface measure on sub-Riemannian manifolds]
{Surface measure on, and the local geometry of, sub-Riemannian manifolds}
\author[S.\, Don]
{Sebastiano Don}
\address[Sebastiano Don]{Mathematisches Institut, Sidlerstrasse 12, 3012 Bern, Switzerland}
\email[]{sebastiano.don@unibe.ch}
\author[V.\, Magnani]{Valentino Magnani}
\address[Valentino Magnani]{University of Pisa, Largo Pontecorvo 5, I-56127 Pisa, Italy.}
\email[]{valentino.magnani@unipi.it}
\date{\today}



\thanks{
	S.D. was partially supported by the Academy of Finland (grant
	288501
	`\emph{Geometry of subRiemannian groups}' and by grant
	322898
	`\emph{Sub-Riemannian Geometry via Metric-geometry and Lie-group Theory}') and by the European Research Council
	(ERC Starting Grant 713998 GeoMeG `\emph{Geometry of Metric Groups}'). S.D. was also partially supported by the Swiss National Foundation grant 200020\_191978.
}

\begin{document}

\begin{abstract}
We prove an integral formula for the spherical measure of hypersurfaces in equiregular sub-Riemannian manifolds.
Among various technical tools, we establish a general criterion for the uniform convergence of parametrized sub-Riemannian distances, and local uniform asymptotics for the diameter of small metric balls.
\end{abstract}
	
\maketitle

\tableofcontents

\section{Introduction} 

Sub-Riemannian manifolds nowadays constitute a wide area of research, related to PDEs, Geometric Analysis, Differential Geometry and Control Theory. Nevertheless, several aspects of their geometry are still far from being understood. Such difficulties already appear in the wide project to develop Geometric Measure Theory in Carnot groups, which are special classes of sub-Riemannian manifolds. To ease our exposition, in the sequel we abbreviate the adjective ``sub-Riemannian'' to ``SR''.

In the present work, we focus our attention on the {\em surface measure} of hypersurfaces embedded in a SR manifold. From a historical perspective, the study of surface measure played an important role in the development of several branches of Mathematics, such as Calculus of Variations, Geometric Analysis, Probability and Geometric Measure Theory. 
Around the half of the twentieth century many works were devoted to find and study the proper notion of surface measure, in view of applications to Calculus of Variations and to the early stages of Geometric Measure Theory. For instance, the questions related to the notion of {\em Lebesgue area} were highly nontrivial and the first works of H. Federer were devoted to the study of this notion of surface area, \cite{FedererPhD1944,FedererSAI,FedererSAII,FedererLA1955}.
The debate on the natural notion of surface measure can be seen from the deep studies of  \cite{Cesari1956,RadoLA1948}.
In the same period, sets of finite perimeter first appeared in the work of R. Caccioppoli \cite{CaccioppoliColl1953} and the theory was developed by E. De Giorgi through his celebrated rectifiability theorem, \cite{DeGiorgi1955}.
The rectifiability of the reduced boundary is crucial to establish the important formula relating perimeter measure to the Hausdorff measure of the reduced boundary.

In the non-Euclidean framework of stratified groups, new theoretical tools are ne\-cessary. A key result is the asymptotically doubling property of the perimeter measure in metric spaces, established by L. Ambrosio \cite{Amb01,Amb2002}, that lead to a rectifiability theorem for sets of finite perimeter in some classes of nilpotent Lie groups, proved by B. Franchi, R. Serapioni and F. Serra Cassano, \cite{FSSC01,FSSC5}, see also \cite{Marchi2014}. The area formulas for the perimeter measure in these papers were corrected in \cite{Mag31ArXiv}, see also \cite{Mag31}. The correction subsequently appeared in \cite{FSSC15} for the case of special symmetric distances.
In all of these works, the same measure-theoretic area formula \cite{Mag30ArXiv,Mag30} was used, see also \cite{LecMag21} for a systematic study of measure-theoretic area formulas.  
The general representation of the perimeter measure with respect to the spherical measure arising from any homogeneous distance was established in \cite{Mag31}. 
Concerning the rectifiability problem in stratified groups, more recent results can be found in \cite{DonLDMV2022}. 

Area formulas for the spherical measure in general homogeneous groups and for higher codimensional smooth submanifolds were obtained in \cite{Magnani2019Area,Mag22RS}, see also references therein. As clarified in these papers, one of the difficulties is to establish negligibility results for ``points of lower degree''. In fact, even for smooth submanifolds the theory of the surface area in homogeneous groups is still far from being complete.

Extending further the study of the surface measure from homogeneous groups to SR manifolds leads to additional difficulties, as the lack of a group operation and of global dilations. We overcome these issues through the study of some interesting metric properties of equiregular SR manifolds, along with their metric tangent spaces.

Our aim is to find an explicit formula that relates the perimeter measure of a bounded, $C^1$ smooth and open set of a SR manifold to the spherical measure of its boundary. The fascinating aspect of this question is that the use of the SR distance to construct the spherical measure naturally lets the geometry of the SR manifold enter the question. In many respects the present work can be seen as a continuation of the study started in \cite{AGM15BVSubRiem}, where the coordinate-free notion of sub-Riemannian perimeter measure was introduced.
The spherical measure in sub-Riemannian manifolds was previously studied in \cite{GheJea2014,GheJea2015,GheJea201516}, where a number of technical difficulties were overlooked, especially in the application of differentiation theorems for measures.

We consider an oriented $\n$-dimensional SR manifold $M$ with a metric $g$ on a distribution of $\m$-planes defined by a family $\cD$ of {\em horizontal vector fields}. The  oriented ``volume measure'' is assigned through an everywhere nonvanishing $\n$-form $\omega$, hence we get a sub-Riemannian measure manifold $(M,\cD,g,\omega)$, where the perimeter measure  $\|D_{\omega,g}\bu_E\|$ of a measurable set $E\subset M$ can be defined in intrinsic terms, see Section~\ref{sect:notions and results} for more details.

For an open set $\Omega$ with $C^1$ smooth boundary $\partial\Omega$, our first result is an integral representation of the perimeter measure
\beq\label{eq:POmega}
\|D_{\omega,g}\bu_\Omega\|(U)=\int_{\partial \Omega\cap U}\|\omega\|_{\overline g}\, |\nu_\mathcal D|_{\overline g}  \,\der \sigma_{\overline g},
\eeq
where we have denoted by $\nu_{\mathcal D}$ the {\em horizontal normal} to $\partial\Omega$ (Definition~\ref{def:hnor}) and $\overline g$ is any Riemannian metric that extends the sub-Riemannian metric $g$. Formula \eqref{eq:POmega} is established in Theorem~\ref{th:formulaperimetro}. The Riemannian surface measure associated with $\partial \Omega$ is denoted by $\sigma_{\overline g}$ and $U\subset M$ is any open subset.
We point out that the proof of \eqref{eq:POmega} requires a general version of the Riemannian divergence theorem with respect to a volume form, see Theorem~\ref{th:divergence}. 
Formula \eqref{eq:POmega} has also an independent interest, since it links the sub-Riemannian perimeter to the Riemannian surface measure $\sigma_{\overline g}$, taking into account the fixed volume form $\omega$. 
The left-hand side of \eqref{eq:POmega} only needs the SR metric $g$, hence the right-hand side does not depend on its extension $\overline g$
and actually motivates the natural definition of {\em sub-Riemannian surface measure} for any smooth hypersurface $\Sigma\subset M$, according to Definition~\ref{def:SRmeas}.
For any Borel set $A\subset \Sigma$, we define
\begin{equation}
	\sigma^{SR}_\Sigma(A)=\int_{\Sigma\cap A} \|\omega\|_{\bar g}\, |\nu_{\mathcal D}|_{\bar g}\,\dd \sigma_{\bar g}.
\end{equation}
The independence of the extension $\bar g$ justifies the slight abuse of notation, denoting by the same symbol $g$ any Riemannian metric that extends the sub-Riemannian metric. 

To use the spherical measure, we focus our attention on equiregular SR manifolds, whose Hausdorff dimension $Q$ has an explicit formula \cite{Mitchell85}. 
Our main result is the following geometric representation of the {\em spherical Federer density} $\cs^{Q-1}$ of $\sigma_\Sigma^{SR}$, namely 
\begin{equation}\label{eq:mainintro}
	\cs^{Q-1}(\sigma^{SR}_\Sigma, q)=\|\omega(q)\|_g\,\beta_{d,g}(\nu_\cD(q)).
\end{equation}
The number $\beta_{d,g}(\nu_\cD(q))$ is the spherical factor, which is described below. Equality \eqref{eq:mainintro} is proved through a blow-up process that also involves the ambient space, since the SR manifold is not homogeneous with respect to local dilations. As a result, the proof of \eqref{eq:mainintro} requires a {\em double blow-up}: the one of the SR manifold and the one of the hypersurface. The first blow-up corresponds to the well known {\em nilpotent approximation} of the SR manifold, representing the metric tangent space of $M$ at the blow-up point $q$, see Section~\ref{sect:nilpapp} for more information.

The left-hand side of \eqref{eq:mainintro} is the {\em spherical Federer density} of $\sigma_\Sigma^{SR}$ at $q$ (Definition~\ref{def:FedDens}). Such ``density'' was first introduced in \cite[Definition~5]{Mag30} to establish a measure-theoretic area formula for the spherical measure (\cite[Theorem~11]{Mag30}), which represents our bridge to the sub-Riemannian area formula \eqref{eq:SRareaperimeter}.
The {\em spherical factor} $\beta_{d,g}(\cdot)$ is a ``pointwise geometric invariant'' of both the SR manifold and the hypersurface $\Sigma$, which is related to the nilpotent approximation of $M$ at $q$ and to the horizontal normal $\nu_\cD(q)$ of $\Sigma$ at $q$.
It can be seen as the maximal area of the intersection between the orthogonal space to $\nu_\cD(q)$ with the sub-Riemannian unit ball in the nilpotent approximation of the SR manifold at $p$. 

For homogeneous groups we have a simpler definition of spherical factor, due to the homogeneity of the ambient space, therefore only the homogeneous tangent space to the submanifold appears (\cite[Definition~7.6]{Magnani2019Area}).
In broad terms, the spherical factor is a sort of ``renormalizing constant'' for the spherical measure. 
To give a simple idea, in the Euclidean space it coincides with the constant $\omega_{\n-1}$ appearing in the definition of the $(\n-1)$-dimensional Hausdorff measure, that is the $(\n-1)$-dimensional Lebesgue measure of the Euclidean unit ball in $\R^{\n-1}$.

However, the formal definition of spherical factor in SR manifolds was not easy to conceive (Definition~\ref{def:beta}). We could also imagine other definitions, like considering sub-Riemannian balls with center close to the blow-up point $q$ and one cannot exclude other possible equivalent definitions. Somehow unexpectedly the definition of spherical factor in SR manifolds came after the proof of the ``double blow-up'', which was obtained by taking a special system of coordinates. 
From \eqref{eq:mainintro} one may deduce a priori that the spherical factor $\beta_{d,g}(\cdot)$ is well defined on the horizontal directions of $TM$ and it depends on a number of mathematical objects, like the sub-Riemannian distance, the Riemannian metric and the hyperplane orthogonal to $\nu_\cD$. However, this information is not enough to find its general explicit formula in Definition~\ref{def:beta}. An ``invariance property'' is necessary and this is provided 
by the next result, proved in Theorem~\ref{th:intrinsicfederer}.

\begin{theorem}[Change of exponential coordinates of the first kind]\label{th:intrinsicfedererIntro}
	Let $(M,\mathcal D,g,\omega)$ be a sub-Riemannian measure manifold and denote by the same symbol $g$ a Riemannian metric on $M$ that extends the sub-Riemannian metric. We assume that $p\in M$ is a regular point and consider two privileged orthonormal frames $\bX=(X_1,\dots, X_\n)$ and $\bY=(Y_1,\dots, Y_\n)$ in an open neighborhood $W$ of $p$. 
	According to \eqref{eqdef:exponentialcoordinates1}, we introduce the exponential coordinates of the first kind  $F_{p,\bX}, F_{p,\bY}\colon V\to W$ associated with $\bX$ and $\bY$ respectively, around $p$. The set $V\subset\R^\n$ is an open neighborhood of $0\in\R^\n$.
	The frames $\widehat\bX^p=(\widehat X_1^p,\dots, \widehat X^p_\n)$ and $\widehat\bY^p=(\widehat Y_1^p,\dots, \widehat Y_\n^p)$ denote the nilpotent approximations of $\bX$ and $\bY$ at $p$, see Definition~\ref{def:nilpApprox}. 
	Then the following facts hold:
	\begin{itemize}
		\item[(i)] the family of maps $\delta_{1/\varepsilon}\circ F_{p,\bY}^{-1}\circ F_{p,\bX}\circ \delta_\varepsilon\colon V\to \mathbb R^\n$ uniformly converges to the restriction of a linear Euclidean isometry $\widehat L\colon \R^\n\to\R^\n$ as $\varepsilon \to 0$;
		\item[(ii)] we have $\widehat L=\dd(F_{p,\bY}^{-1}\circ F_{p,\bX})(0)$ and the matrix associated with $\widehat L$ is block diagonal;
		\item[(iii)] if we denote by $\widehat d_{p,\bX}$ and $\widehat d_{p,\bY}$ the sub-Riemannian distances associated with $\widehat\bX^p$ and $\widehat\bY^p$,
		respectively, then $\widehat d_{p,\bY}(\widehat L(x),\widehat L(y))=\widehat d_{p,\bX}(x,y)$ for every $x,y\in \mathbb R^\n$.
	\end{itemize}
\end{theorem}

The notation of this theorem is introduced in Sections~\ref{subsec:expSR} and \ref{sect:nilpapp}.
Theorem~\ref{th:intrinsicfedererIntro}  implies that the spherical factor does not depend on the choice of 
the exponential coordinates of the first kind, see Corollary~\ref{cor:betaintrinsic}.
Now we state the double blow-up that leads to the key equality \eqref{eq:mainintro}.
Its proof corresponds to that of Theorem~\ref{th:main}.

\begin{theorem}[Double blow-up]\label{th:mainIntro} 
	Let $(M,\mathcal D, g, \omega)$ be a sub-Riemannian measure manifold. We assume that $\Sigma\subset M$ is an oriented $C^1$ smooth hypersurface with orienting unit normal $\nu$ on $\Sigma\cap T$ and $T\subset M$ is an open neighborhood of $p\in \Sigma$. 
	We also denote by the same symbol $g$ a Riemannian metric on $M$ that extends the given sub-Riemannian metric and we consider the associated SR surface measure $\sigma_\Sigma^{SR}$.
	If $p$ is both a regular point of $M$ and a noncharacteristic point of $\Sigma$, then
	\beq\label{eq:theta=betaIntro}
	\cs^{Q-1}(\sigma_\Sigma^{SR}, p)=\|\omega(p)\|_g\, \beta_{d,g}(\nu_\cD(p)),
	\eeq
	where $\nu_\cD(p)$ denotes the horizontal normal at $p\in\Sigma$.
\end{theorem}
Characteristic points represent the ``singular points'' of $\Sigma$ (Definition~\ref{def:hnor}) and will 
be discussed right after Theorem~\ref{thm:areaS}. Regular points are introduced in Definition~\ref{def:equiregular}.
To prove Theorem~\ref{th:mainIntro}, more difficulties are hidden, since the formula for the spherical Federer density $\cs^{Q-1}$ needs uniform asymptotics for the diameters of sub-Riemannian balls.
More precisely, another result is necessary, showing that the diameters of a family
of sub-Riemannian balls with suitably small radius $r>0$ 
are ``uniformly close to $2r$'' in a neighborhood of a regular point. 

\begin{theorem}[Uniform estimates of ``small'' diameters]\label{teo:diameterIntro}
	Let $p$ be a regular point of a sub-Riemannian manifold $M$.
	Then there exists a neighborhood $T\subset M$ of $p$ such that for every $0<\ep<1$ there exists a radius $r_\ep>0$ such that
	\begin{equation}\label{eq:asymptDiam}
	2r(1-\ep)	\le {\rm diam}(B(q,r))\le 2r
\end{equation}
for every $q\in T$ and $0<r<r_\ep$.
\end{theorem}
The uniform asymptotics \eqref{eq:asymptDiam} are established in Theorem~\ref{teo:diameter} and we believe they have an independent interest in the study of equiregular sub-Riemannian manifolds. The proof of  \eqref{eq:asymptDiam} in turn needs two additional results. The first one is a {\em uniform nilpotent approximation} with respect to the blow-up point, which extends the well-known nilpotent approximations at fixed regular points, for which a wide literature is available, see Section~\ref{sect:nilpapp} for more information.

\begin{theorem}[Uniform nilpotent approximation]\label{thm:unifblowupIntro}
	Let $(M,\mathcal D, g)$ be a sub-Riemannian manifold and let $p\in M$ be a regular point.
	We consider a privileged orthonormal frame $\bX=(X_1,\dots, X_\n)$ in an open neighborhood of $p\in M$ 
	and denote by $F_\bX$ a system of uniform exponential coordinates on a fixed open set $U$ containing $p$ (Definition~\ref{def:unifexpcoord}).
	We define the vector fields in local coordinates $\widetilde X_i^q= (F_\bX(q,\cdot)^{-1})_\ast X_i$, along with the rescaled vector fields $\widetilde X_i^{q,r}=r^{w_i}(\delta_{1/r})_*\widetilde X_i^q$ for $i=1,\ldots,\n$, $q\in U$ and $r>0$, see \eqref{eq:Xcappuccio}.
	Then for every bounded open set $A\subset\R^\n$ the following statements hold.
	\begin{enumerate}
		\item 
		The rescaled frame $\widetilde \bX^{q,r}=(\widetilde X^{q,r}_1,\ldots,\widetilde X^{q,r}_\n)$ converges to the nilpotent approximation $\widehat \bX^q=(\widehat X^q_1,\ldots,\widehat X^q_\n)$ (Definition~\ref{def:nilpApprox}) on the subset $A\subset\R^n$ in the $C^\infty_{\rm loc}$-topology as $r\to0$, uniformly with respect to $q$ which varies in any compact set of $U$.
		\item 
		The rescaled horizontal frame $\widetilde \bX^{q,r}_h=(\widetilde X^{q,r}_1,\ldots,\widetilde X^{q,r}_\m)$ induces
		a local distance $\widetilde d_q^r$ which converges to $\widehat d_q$ (Definition~\ref{def:nilpApprox}) in $L^\infty(A\times A)$ as $r\to 0$, uniformly as $q$ varies in any compact set of $U$.
	\end{enumerate}
\end{theorem}
We have denoted by $\widehat d_q\colon \R^n\times\R^\n\to[0,+\infty)$ the tangent sub-Riemannian distance (Definition~\ref{def:nilpApprox})
and by $\widetilde d^r_q$ the local distance (Definition~\ref{def:balls}) induced by the rescaled frame $\widetilde \bX^{q,r}$.
The uniform nilpotent approximation is proved in Theorem~\ref{thm:unifblowup}, where more details are added in the statement.
Theorem~\ref{thm:unifblowupIntro} is somehow considered known to the experts, but we were unable to find its proof.
We notice for instance that the estimate (85) in \cite{Bell96} follows from Theorem~\ref{thm:unifblowupIntro}(ii).
According to \cite[Section~8]{Bell96}, we can actually think of
Theorem~\ref{thm:unifblowupIntro}(ii) as a uniform Gromov-Hausdorff convergence of the SR manifold to the metric tangent space.
When a regular point $p\in M$ is fixed (Definition~\ref{def:equiregular}), Theorem~\ref{thm:unifblowupIntro} proves that the rescaled distances $\widetilde d^r_q$ uniformly converge to the tangent sub-Riemannian distance $\widehat d_q$ (Definition~\ref{def:nilpApprox}) on compact sets of their domains
and also uniformly as $q$ varies in a compact neighborhood of $p$. 

A nontrivial technical result behind the uniform nilpotent approximation is a general uniform convergence theorem for families of sub-Riemannian distances, established in Theorem~\ref{th:DVGen}, see Section~\ref{sec:blowup} for more information.

Somehow surprisingly, the uniform convergence of the rescaled distances $\widetilde d_q^r$ does not immediately imply the uniform diameter estimate of Theorem~\ref{teo:diameterIntro}.
The main issue is that the local sub-Riemannian distance (Definition~\ref{def:balls}) relative to a bounded open set may be larger than the sub-Riemannian distance on the whole manifold. Only at sufficiently small scale, the two distances coincide, as proved in Proposition~\ref{prop:localdist}. 
To apply this proposition, we have to show that the radii of the sub-Riemannian balls contained in the ``moving exponential local charts'' {\em do not degenerate as the center of the sub-Riemannian ball and the chart vary}.
In other words, we have to guarantee the existence of a 
``common, moving sub-Riemannian ball with fixed radius'', where the uniform convergence of the sub-Riemannian distances takes place.
Exactly at this point the second theorem to show \eqref{eq:asymptDiam} appears. 
The existence of the ``uniform radius'' can be actually established for a general ``topological exponential mapping'' taking values in a length metric space.

\begin{theorem}[Topological existence of uniform radius]\label{thm:TopolUniformradiusIntro}
Let $M$ be a length metric space and let $p\in M$. Let $T\subset M$ be an open neighborhood of $p$
	and let $A\subset \R^\n$ be an open neighborhood of $0$. 
	We consider a mapping 
	$E\colon T\times A\to M$ such that
	\begin{enumerate}
		\item\label{it1contIntro}
		$E$ is continuous,
		\item\label{it2Eq0Intro}
		$E(q,0)=q$ for every $q\in T$,
		\item\label{it3homIntro}
		the mapping $E(q,\cdot)\colon A\to E(q,A)$ is a homeomorphism for every $q\in T$.
	\end{enumerate}
	Then there exist a bounded open neighborhood $V\subset A$ of $0$, an open neighborhood $U\subset T$ of $p$ such that the function
	\begin{equation}\label{eq:r_qpositiveIntro}
		U\ni q\mapsto R(q)=\sup\{t>0: B(q,t)\subset E(q,V) \}\in(0,+\infty)
	\end{equation}
	is well defined and lower semicontinuous.
	In particular, there exist $r_0>0$ and $\ep_0>0$ such that 
	$B(q,r_0)\subset E(q,V)$ for every $q\in B(p,\ep_0)$.
\end{theorem}
The arguments to prove this theorem also rely on a suitable use of degree theory, see Theorem~\ref{thm:TopolUniformradius}.
Combining the measure theoretic area formula of Theorem~\ref{th:sphericalarea} and the double blow-up
of Theorem~\ref{th:main}, we finally obtain the area formula \eqref{eq:SRareaperimeter}. More details about the application of Theorem~\ref{th:sphericalarea} are given in Remark~\ref{rem:metricarea}.

\begin{theorem}[Area formula]\label{thm:areaS}
Let $(M,\cD,g,\omega)$ be an equiregular sub-Riemannian measure manifold of Hausdorff dimension $Q$ and denote also by $g$ a Riemannian metric that coincides with the sub-Riemannian metric on horizontal directions.
Let $\Sigma\subset M$ be an oriented hypersurface of class $C^1$ embedded in $M$ and let $A\subset \Sigma$ be a Borel set containing an $\cS^{Q-1}$ negligible subset of characteristic points. 
If $\nu$ is a continuous unit normal field on $\Sigma$, then we have
\beq\label{eq:SRareaperimeter}
\sigma^{SR}_\Sigma(A)=\int_{A\cap\Sigma} \|\omega(q)\|_g\,\beta_{d,g}(\nu_\cD(q))\,d\cS^{Q-1}(q),
\eeq
where the spherical measure is introduced in Definition~\ref{d:sizephi}.
\end{theorem}
We notice that the set of characteristic points is $\cS^{Q-1}$ negligible for $C^1$ smooth hypersurfaces embedded in stratified groups, \cite{Mag5}. The extension of this result to SR manifolds is not straightforward, since in general they are not locally bi-Lipschitz equivalent to stratified groups, \cite{Bell96,Pansu82,Var1981}, see also \cite{LeDOttWar2014}. It is then natural to deserve further study to the negligibility issue.

As a simple consequence of the previous area formula, for a $C^1$ smooth open set $\Omega$ with boundary $\partial\Omega$, formula \eqref{eq:POmega} combined with \eqref{eq:SRareaperimeter} yields
\beq\label{eq:POmegaSpherical}
\|D_{\omega,g}\bu_\Omega\|(A)=\int_{A\cap \partial \Omega} \|\omega(q)\|_g\,\beta_{d,g}(\nu_\cD(q))\,d\cS^{Q-1}(q),
\eeq
for every Borel set $A\subset M$ such that the set of characteristic points in the intersection $\partial\Omega\cap A$ is $\cS^{Q-1}$ negligible. Notice that \eqref{eq:POmegaSpherical} also extends 
the area formula for perimeters in stratified groups equipped with a general volume form, in place of the standard Haar measure.

\section{Some basic notions and known facts}\label{sect:notions and results}

In this section we introduce {\em sub-Riemannian measure manifolds}, along with some basic notions.
We also introduce the important notion of exponential coordinates of the first kind, that will be useful
in the sequel.

To define a sub-Riemannian manifold without referring to local frame of vector fields,
we use the notion of Euclidean vector bundle, that is a vector bundle 
equipped with a smooth scalar product.

\begin{definition}[Sub-Riemannian structure]\label{def:SRstructure}
	Let $M$ be a smooth, connected manifold and assume that 
	$f\colon{\bf U}\to TM$ is a morphism of vector bundles, where ${\bf U}$ is a Euclidean
	vector bundle with base $M$. As a morphism, $f$ is smooth, linear on fibers and verifies
	$f_p({\bf U}_p)\subset T_pM$ for each $p\in M$, where ${\bf U}_p$ is the fiber of ${\bf U}$ at $p$. We say that $(f,{\bf U})$ is a {\em sub-Riemannian structure} on $M$.
	The family of {\em horizontal vector fields} are 
	\begin{equation}\label{def:calD}
	\cD\coloneqq\left\{f\circ\sigma : \sigma\in\Gamma({\bf U})\right\},
	\end{equation}
	where $\Gamma({\bf U})$ is the module of sections of ${\bf U}$.
	We assume the {Chow's condition} on $\cD$: 
	\begin{equation}\label{eq:Hormander}
	\Lie_p(\cD)=T_pM\qquad \text{for all}\; p\in M.
	\end{equation}
	For each $p\in M$, we set the fiber 
	$$
	\cD_p=\set{X(p)\in T_pM:\, X\in\cD}.
	$$
	We define the function 
	$G_p\colon T_pM\to[0,+\infty]$ as 
	\[
	G_p(v)=\begin{cases} 
	\min\{|u|_p^2: v=f(u), u\in\bU_p \}& v\in \mathcal D_p, \\
	+\infty & \text{\rm otherwise}.
	\end{cases}
	\]
	One may easily notice that $\sqrt{G_p}\colon \cD_p\to[0,+\infty)$ is a norm
	that satisfies the parallelogram identity,
	hence $G_p$ is a quadratic form.
	The family of scalar products
	$g_p\colon \cD_p\times\cD_p\to \R$
	such that $G_p(v)=g_p(v,v)$ for
	all $p\in M$ and $v\in \cD_p$
	defines a {\em sub-Riemannian metric} $g$ on $M$. The triple 
	$(M,\cD,g)$ is called {\em sub-Riemannian manifold}.
\end{definition}

For $q_1,q_2\in M$, we introduce the family
$\cA_{q_1,q_2}$ of all AC curves 
$\gamma\colon [0,1]\to M$ such that 
$\gamma(0)=q_1$, $\gamma(1)=q_2$ and
\beq\label{eq:gammaD}
\dot\gamma(t)\in\cD_{\gamma(t)} 
\eeq
for a.e.\ $t\in[0,1]$.
Due to the celebrated Chow's theorem,
$\cA_{q_1,q_2}\neq\emptyset$ for
all $q_1,q_2\in M$.
As a consequence, the infimum
\[
d(q_1,q_2)=\inf\set{\int_0^1 \sqrt{g_{\gamma(t)}(\dot\gamma(t),\dot\gamma(t))}dt:\,\gamma\in\cA_{q_1,q_2} }
\]
is always a well defined real number. 
Actually, it can be shown that $d$ is a distance on $M$, well known as {\em sub-Riemannian distance}. 
Any AC curve $\gamma$ that satisfies
\eqref{eq:gammaD} for a.e.\ $t$ is called {\em horizontal curve}.
The closed ball and the open ball in $M$ are denoted by
\beq \label{key}
\B(q,r)=\{z\in M: d(q,z)\le r\}\quad \text{and}\quad 
B(q,r)=\{z\in M: d(q,z)< r\},
\eeq 
respectively.

The sub-Riemannian distance can be also introduced with respect to a family of linearly independent vector fields of $\R^\n$.

\begin{definition}[Sub-Riemannian distance with respect to vector fields]\label{def:SRdX_j}
	Let $\Omega\subset\R^\n$ be a connected open set, $\m<\n$  and let $\bX=(X_1,\ldots,X_\m)$ be an ordered collection of 
	everywhere linearly independent smooth vector fields on $\Omega$. We also assume that the iterated Lie brackets of 
	these vector fields, up to some degree of iteration $s$, span $\R^\n$ at every point $x\in\Omega$.
	Then we choose $x,y\in\Omega$ and consider the family $\cF(x,y)$ of all absolutely continuous curves 
	$\gamma:[0,1]\to\Omega$ such that $\gamma(0)=x$, $\gamma(1)=y$ and 
	\beq\label{eq:gammah}
	\gamma'(t)=\sum_{j=1}^\m h_j(t)\,X_j(\gamma(t)) 
	\eeq
	for a.e.\ $t\in[0,1]$. The sub-Riemannian distance $d(x,y)$ between $x$ and $y$ in $\Omega$ is then 
	\[
	\inf\set{\|h\|_\infty: \gamma\in\cF(x,y)\;\text{and satisfies \eqref{eq:gammah} a.e.} },
	\]
	where $\|h\|_\infty$ is the $L^\infty$ norm of the measurable function $t\to\sqrt{\sum_{j=1}^{\m}h_j(t)^2}$.
\end{definition}

\begin{rmk}
The sub-Riemannian distance with respect to a sub-Riemannian metric and the one 
with respect to vector fields (Definition~\ref{def:SRdX_j}) are clearly
related, see Proposition~\ref{prop:localdist}. This is the case when the vector fields are considered orthonormal
with respect to the sub-Riemannian metric.
\end{rmk}

\begin{rmk}\label{rem:equivSRX_j}
It is easy to notice that the distance of Definition~\ref{def:SRdX_j} is equivalent to the one where we consider time minimizing
absolutely continuous curves $\gamma:[0,T]\to\Omega$ satisfying \eqref{eq:gammah}, connecting $x$ with $y$ and such that
$\|h\|_\infty\le 1$ a.e.  
\end{rmk}

\begin{definition}[Equiregular sub-Riemannian manifolds and regular points]\label{def:equiregular}
Let us consider the sub-Riemannian structure $(f,{\bf U})$ on a smooth, connected manifold $M$.
	If $\cD$ is the associated family of horizontal vector fields, we 
	recursively define $\cD^1=\cD$ and
	\[
	\cD^{k+1}=\cD^k+[\cD^k,\cD]
	\]
	Setting  $\cD_p^k=\set{X(p)\in T_pM:\, X\in\cD^k}$, we obtain a flag at $p\in M$.
	Due to the Lie bracket generating condition, there exists $k_p\in\N$ such that 
	\[
	\cD^1_p\subset \cD^2_p\subset\cdots\subset\cD_p^{k_p}=T_pM.
	\]
	If $g$ is a sub-Riemannian metric on $M$, compatible with $(f,{\bf U})$, 
	we say that $(M,\mathcal D, g)$ is an \emph{equiregular sub-Riemannian manifold of step $s\in \mathbb N$} if $\mathcal D^s_p=T_pM$ for every $p\in M$ and there exist some positive integers $\n_1, \n_2,\dots, \n_s$ such that $\dim \mathcal D^k_p=\n_k$ for every $k=1,\dots, s$ and every $p\in M$.
	A {\em regular point} $p\in M$ has an open neighborhood $\mathcal U\subset M$ that
	is also an equiregular sub-Riemannian manifold,
	when equipped with the restriction of the sub-Riemannian structure.
\end{definition}

The triple $(M,\cD,g)$ always denotes an $\n$-dimensional, smooth and connected SR manifold of step $s$. A fixed Riemannian metric on $M$ that extends $g$ is also understood and it is denoted by the same symbol, unless otherwise stated. The sub-Riemannian distance associated with the sub-Riemannian manifold $(M,\cD,g)$ is denoted by $d$. 

For equiregular sub-Riemannian manifolds, we also set $\m=\n_1=\dim\cD_p$ for every $p\in M$, that is called {\em the rank of $\cD$}. By definition of $\cD$, for every $p\in M$ there exists a neighborhood $U$ of $p$ and smooth vector fields $X_1,\dots,X_\m$ on $U$ such that $\mathcal D_q=\mathrm{span} \{X_1(q),\dots, X_\m(q)\}$ for every $q\in U$. We say that $(X_1,\dots,X_\m)$ is a {\em local frame} of $\cD$, that is also Lie bracket generating. 

\begin{definition}[Length of iterated Lie brackets]\rm
	Let $(X_1,\ldots,X_\m)$ be a local frame for $\cD$.
	Whenever a multi-index $I=(i_1,\ldots,i_k)\in\{1,2,\ldots,\m\}^k$ 
	is fixed, we say that
	\[
	X_I=[\ldots[[X_{i_1},X_{i_2}],X_{i_3}],\ldots],X_{i_k}]
	\] 
	is a vector field of {\em length $k$} with respect to $(X_1,\ldots,X_\m)$.
\end{definition}

\begin{definition}
	A \emph{sub-Riemannian measure manifold} is a quadruple $(M,\mathcal D,g,\omega)$ such that $(M, \mathcal D, g)$ is an oriented sub-Riemannian manifold and $\omega$ is a positive volume form on $M$, namely a non-vanishing smooth $\n$-form on $M$ such that $\int_M f \omega\ge 0$, 
	for any nonnegative  $f\in C_c(M)$. We say that $(M,\mathcal D, g, \omega)$ is \emph{equiregular}, if so is $(M,\mathcal D, g)$.
\end{definition}

The following definition can be found e.g.\ in \cite[Definition 10.45]{ABB20}.

\begin{definition}[Privileged frame]\label{def:privcoord}
	Let $(M,\mathcal D, g)$ be an equiregular sub-Riemannian manifold of step $s$ and let $U\subset M$ be an open set. We say that a frame of smooth vector fields $(X_1,\dots, X_\n)$ is {\em privileged with respect to $\mathcal D$ on $U$} if the following conditions are satisfied.
	\begin{itemize}
		\item[(i)] For every $q\in U$, the vectors $(X_1(q),\dots, X_\m(q))$ form
		a basis for $\mathcal D_q$.
		\item[(ii)]
		For every $q\in U$ and $1\le j\le s$, the vector
		fields $X_{\n_{j-1}+1},\ldots,X_{\n_j}$ are iterated Lie brackets of $X_1,X_2,\ldots,X_\m$ of length $j$, where 
		we have set $\n_0=0$ and $\n_j=\dim\cD^j_q$ for $j=1,\ldots,s$.
		\item[(iii)] For every $q\in U$ and $j=1,\dots, s$, the vectors 
		$(X_{1}(q),\dots, X_{\n_j}(q))$ form a basis of $\cD^j(q)$.
	\end{itemize} 
\end{definition}

\begin{rmk}
	A standard argument shows that, locally, we can always find a privileged frame for an equiregular sub-Riemannian manifold. Indeed, consider a local frame $(X_1,\dots, X_\m)$
	of $\mathcal D$ and choose $X_{\m+1}, \dots, X_{\n_2}$ among the Lie brackets of $X_1, \dots, X_\m$
	such that $(X_1,\ldots,X_{\n_2})$ is a local frame of $\cD^2$.
	The analogous argument can be repeated for the subsequent $\cD^j$'s,
	to get the privileged frame.
\end{rmk}

\begin{lemma}\label{lemma:estensioneadattata}
	Let $(M,\mathcal D, g)$ be an equiregular sub-Riemannian manifold, let $p\in M$, and let $(v_1,\dots,v_\m)$ be an orthonormal basis for $\mathcal D_p$. We denote by $g$ also a  Riemannian metric that extends the fixed sub-Riemannian metric on $M$. Then there exist a neighborhood $U$ of $p$ and a privileged orthonormal frame $(X_1,\dots,X_\n)$ on $U$
	such that $X_i(p)=v_i$ for every $i=1,\dots, \m$.
\end{lemma}
\begin{proof}
	Since $\cD$ is locally spanned by smooth horizontal vector fields, we can first find linearly independent horizontal vector fields $Y_i$ on a neighborhood of $p$ such that $Y_i(p)=v_i$ for all $i=1,\ldots,\m$. Then we apply the Gram--Schmidt algorithm to these vector fields in order to obtain an orthonormal horizontal frame 
	$X_1,\ldots,X_\m$ such that $X_i(p)=v_i$ with $i=1,\ldots,\m$.
	Then we choose a frame $(Y_{\m+1},\ldots,Y_{\n_2})$ such that
	$(X_1(q),\ldots,X_\m(q),Y_{\m+1}(q),\ldots,Y_{\n_2}(q))$ is a basis of $\cD^2_q$
	for every $q$ in a neighborhood of $p$. Again the Gram--Schmidt algorithm 
	provides an orthonormal frame $X_{\m+1},\ldots,X_{\n_2}$ such that
	$(X_1(q),\ldots,X_\m(q),X_{\m+1}(q),\ldots,X_{\n_2}(q))$ is an
	orthonormal basis of $\cD^2_q$ for every $q$ in a neighborhood of $p$.
	Repeating this argument we find a neighborhood of $p$ and 
	a privileged frame $(X_1,\ldots,X_\n)$ on $U$ that satisfy our claim.
\end{proof}

\subsection{Exponential coordinates and local sub-Riemannian distance}\label{subsec:expSR}
We begin the section by introducing the notion of exponential coordinates with respect to a privileged frame of vector fields.

\begin{definition}[Exponential map]
	Given a smooth vector field $X$ on $M$ and a compact set $K$, there exists $\delta>0$ such that the unique solution of the Cauchy problem
	\[
	\begin{cases}
	\dot \gamma_p=X\circ \gamma_p\\
	\gamma_p(0)=p,
	\end{cases}
	\]
	is well defined on $[-\delta,\delta]$. The exponential map of $X$ is then defined by
	\[
	\exp(tX)(p)=\gamma_p(t)
	\]
	for every $t\in [-\delta,\delta]$ and $p\in K$.
\end{definition}

\begin{definition}[Exponential coordinates of the first kind]\label{def:unifexpcoord}
We fix a regular point $p\in M$ of a sub-Riemannian manifold $(M,\cD,g)$ and consider a privileged frame $\bX=(X_1,\dots, X_\n)$ with respect to $\cD$  on a neighborhood $W$ of $p$. For each $q\in W$ we may find a smooth diffeomorphism 
$F_{q,\bX}\colon V_q\to F_{q,\bX}(V_q)\subset W$ defined as
\beq\label{eqdef:exponentialcoordinates1}
F_{q,\bX}(x)=\exp(x_1X_1+\dots+x_\n X_\n)(q),
\eeq
for some open set $V_q\subset \R^\n$ containing the origin. 
We say that
$(x_1,\ldots,x_\n)$ of \eqref{eqdef:exponentialcoordinates1},
corresponding to $F_{q,\bX}^{-1}$, are 
{\em exponential coordinates of the first kind centered at $q$.}
We may also choose an open neighborhood $V\subset \R^\n$ of $0$ 
and an open neighborhood $U\subset W$ of $p$ such that
the smooth map $F_\bX\colon U\times V\rightarrow W$, 
\beq\label{eqdef:exponentialcoordinates}
F_\bX(q,x)=\exp(x_1X_1+\dots+x_\n X_\n)(q)
\eeq
is well defined on $U\times V$
and, by standard ODE arguments, $F_\bX(q,\cdot)\colon V\to F_\bX(\set{q}\times V)$
is a $C^\infty$ diffeomorphism for every $q\in U$. 
In this case, we say that $F_\bX$ represents a system of \emph{uniform exponential coordinates of the first kind relative to the frame $\bX=(X_1,\ldots,X_\n)$}.
\end{definition}

The exponential coordinates of the first kind can be naturally associated with a family of dilations as follows. 
We assign the weight $w_i=j$ to a coordinate $x_i$ if $\n_{j-1}<i\leq \n_j$, where $\n_i$ is introduced in Definition~\ref{def:privcoord}.
Then for every $r>0$ we define the {\em anisotropic dilation} $\delta_r\colon \R^\n\to \R^\n$ by setting
\begin{equation}\label{df:tangdilations}
	\delta_r x=\sum_{i=1}^\n r^{w_i}x_ie_i,
\end{equation}
where $(e_1,\ldots,e_\n)$ denotes the canonical basis of $\R^\n$.
We say that a function $\phi\colon\R^\n\to\R$ is {\em $\delta$-homogeneous of degree $\alpha>0$} 
if $\phi(\delta_rx)=r^{\alpha} \phi(x)$ for every $x\in\R^\n$ and $r>0$.

\subsubsection{Sub-Riemannian manifold in local coordinates}\label{subsect:inducedSR}

We fix a regular point $q\in M$ of a sub-Riemannian manifold $(M,\cD,g)$ and choose some exponential coordinates $(x_1,\ldots,x_\n)$ of an open and connected neighborhood $A$ of $0$.
Such coordinates are given by the diffeomorphism
$$F_{q,\bX}\colon A\to F_{q,\bX}(A),$$
according to \eqref{eqdef:exponentialcoordinates1}, where clearly $F_{q,\bX}(0)=q$
and $F_{q,\bX}(A)\subset M$ is an open neighborhood of $q\in M$.
We may define the vector fields
\beq\label{eq:localq}
\widetilde X_j^q=(F_{q,\bX})_*^{-1}(X_j)
\eeq
for each $j=1,\ldots,\n$ on $A\subset\R^\n$. The ordered family of vector fields
$(\widetilde X_1^q,\dots, \widetilde X_\n^q)$ is a privileged frame
with respect to the horizontal vector fields 
$$
\widetilde {\mathcal D}_{q,\bX}=(F_{q,\bX})_*^{-1}\cD
$$ 
on $A$. Considering the preimage of the sub-Riemannian metric $g$ on $M$, we have obtained an {\em induced sub-Riemannian manifold} on $A$ defined by the triple 
$$
(A,\widetilde{\mathcal D}_{p,\bX}, F_{q,\bX}^* g).
$$ 
Since $A\subset \R^\n$ is connected, we have obtained a {\em local sub-Riemannian distance} $\widetilde d_q$ {\em on} $A$ with respect to $\widetilde {\mathcal D}_{q,\bX}$ and $F_{q,\bX}^* g$, that is obtained by {\em considering only horizontal curves contained in $A$}.

\begin{definition}[Local distance and local sub-Riemannian ball]\label{def:balls}
We fix a regular point $q$ of a SR manifold $(M,\cD,g)$ and an associated  privileged frame $\bX=(X_1,\ldots,X_\n)$ defined on a neighborhood $W$ of $q$. Let $F_{q,\bX}\colon A\to F_{q,\bX}(A)$ define a system of exponential coordinates of the first kind centered at $q$, where $A\subset\R^\n$ is an open and connected neighborhood of $0\in\R^\n$. 
Let $\widetilde d_q$ be the {\em local sub-Riemannian distance} defined by the frame
$(\widetilde X^q_1,\ldots,\widetilde X^q_\n)$ of \eqref{eq:localq} on the subset $A\subset\R^\n$ (Definition~\ref{def:SRdX_j}).
Then we have the associated {\em local sub-Riemannian balls}
\beq
	\widetilde B_q(x,r)=\set{y\in A:\widetilde d_q(x,y)<r},
\eeq
where $x\in A$ is the center and $r>0$ is the radius.
The corresponding closed ball is 
\beq \label{eq:closedballs}
	\widetilde \B_q(x,r)=\{y\in  A: \widetilde d_q(x,y)\leq r \}.
\eeq
\end{definition}

\begin{rmk}
We stress that the local sub-Riemannian balls depend on the fixed domain $A$ of the exponential coordinates and and also refers to a distance made by curves contained in $A$.
\end{rmk}

Since $\widetilde d_q$ strongly depends on $A$, we may well ask whether it can be related to the sub-Riemannian distance $d$ on $M$. 
More precisely, we may wonder which conditions ensure that the local distance $\widetilde d_q$ on $A$ makes $F_{q,\bX}\colon A\to F_{q,\bX}(A)$ an isometry,
where $F_{q,\bX}(A)\subset M$, where we consider the sub-Riemannian distance $d$ on $M$.

In the next proposition, we show that for all $x,y\in A$ it holds
\begin{equation}\label{eq:isomdG}
	\widetilde{d}_q(x,y)=d(F_{q,\bX}(x),F_{q,\bX}(y)),
\end{equation}
if $A$ is considered ``suitably small''.

\begin{proposition}\label{prop:localdist}
Let $A\subset \R^\n$ be an open and connected neighborhood of $0\in\R^n$ and let
$F_{q,\bX}\colon A\to F_{q,\bX}(A)$ be a system of exponential coordinates of the first kind, centered at a regular point $q$ of a SR manifold $(M,\cD,g)$. 
We fix $r>0$ such that $B(q,4r)\subset F_{q,\bX}(A)$ and set $\widetilde A=F_{q,\bX}^{-1}(B(q,r))\subset A$.
Then the restriction $F_{q,\bX}\colon\widetilde A\to F_{q,\bX}(\widetilde A)$ is an isometry, namely  \eqref{eq:isomdG} holds for every 
$x,y\in\widetilde A$
and we have  $\widetilde A=\widetilde B_q(0,r)$.
\end{proposition}
\begin{proof}
We choose $q_1,q_2\in B(q,r)$ and $\varepsilon\in(0,r)$. There exists a horizontal curve $\gamma\colon [0,1]\to M$ connecting $q_1$ to $q_2$ such that 
	\[
	{\rm length}(\gamma)=\int_0^1 \sqrt{g_{\gamma(t)}(\dot\gamma(t),\dot\gamma(t))}dt< d(q_1,q_2)+\varepsilon.
	\]
By the triangle inequality, one can easily check that
	\[
	d(\gamma(t),q)\leq d(\gamma(t),q_1)+d(q_1,q)< {\rm length}(\gamma)+r< d(q_1,q_2)+\varepsilon+r<4r,
	\]
hence $\gamma(t)\in B(q,4r)\subset $ for all $t\in[0,1]$.
As a consequence, by definition 
of sub-Riemannian manifold $(A, \widetilde{\mathcal D }, F_{q,\bX}^* g)$, the curve $\widetilde \gamma=F_{q,\bX}^{-1}\circ \gamma$ is horizontal 
and contained in $\widetilde A$, connects
the preimages $x=F_{q,\bX}^{-1}(q_1)\in\widetilde A$ and $y=F_{q,\bX}^{-1}(q_2)\in\widetilde A$
and has the same length of $\gamma$. It follows that
\[
\widetilde d_q(x,y)\le {\rm length}(\widetilde \gamma)={\rm length}(\gamma)<d(q_1,q_2)+\ep
\]
and the arbitrary choice of $\ep>0$ yields
\[
\widetilde d_q(x,y)\le d(q_1,q_2)
=d(F_{q,\bX}(x),F_{q,\bX}(y)).
\]
If we consider any horizonal curve $\tilde\gamma\colon[0,1]\to \widetilde A$ connecting $F_{q,\bX}^{-1}(q_1)$ and $F_{q,\bX}^{-1}(q_2)$, then $F_{q,\bX}\circ\widetilde \gamma$ is a horizontal curve of $M$ that connects $q_1$ and $q_2$ and has the same length. As a consequence, the opposite inequality immediately follows,
hence \eqref{eq:isomdG} holds. 
If we take $w\in B(q,r)$, we have $\widetilde d_q(F_{q,\bX}^{-1}(w),0)=d(w,q)<r$, therefore $\widetilde A=F_{q,\bX}^{-1}(B(q,r))\subset \widetilde B_q(0,r) $.
If $x\in \widetilde B_q(0,r)\subset A$, we choose a horizontal curve 
$\widetilde \gamma\colon[0,1]\to A$ connecting $x$ and $0$ such that 
\[
{\rm length}(\widetilde \gamma)<r.
\]
Then $\gamma=F_{q,\bX}\circ \widetilde \gamma$ is a horizontal curve connecting $F_{q,\bX}(x)$ and $q$ with
\[
d(F_{q,\bX}(x),q)\le {\rm length}(\gamma)={\rm length}(\widetilde \gamma)<r.
\]
The opposite inclusion holds and the proof is complete.
\end{proof}

\subsection{Nilpotent approximation}\label{sect:nilpapp}

In geometric terms, the nilpotent approximation corresponds to a metric tangent cone, that can be obtained for equiregular sub-Riemannian manifolds using the well known Gromov-Hausdorff convergence of metric spaces, \cite{Mitchell85}, see also
\cite[Theorem~7.36]{Bell96}. 
There is a huge literature on this topic, which goes back to the works of Rothschild and Stein \cite{RothschildStein76}, Goodman \cite{Goodman1976} and Metivier \cite{Met1976}. We mention for instance the papers \cite{AGM15BVSubRiem, AntLeDNicGol2022LipCCdist, Bell96, Bell97, Jean2001, MarMos2000, MPV2018}, along with the monographs \cite{ABB20} and \cite{Jean2014}.
A systematic study of privileged coordinates and the nilpotent approximations has been developed in the recent papers \cite{ChoPon2019I, ChoPon2019II}, where more references can be found. 
For our approach, it is convenient to consider the nilpotent approximation with respect to exponential coordinates of the first kind. We refer to \cite[Section~2]{MPV2018}, especially Theorem~2.3, Proposition~2.5 and Remark~2.6, that are summerized in the next result.

\begin{theorem}\label{th:tangentspace}
Let $(M, \mathcal D, g)$ be an equiregular sub-Riemannian manifold, let $p\in M$ and let $(X_1,\dots, X_\n)$ be a privileged frame of smooth vector fields on an open set $W$ containing $p$. Let $F_{p,\bX}\colon V_p\to F_{p,\bX}(V_p)\subset W$ be as in \eqref{eqdef:exponentialcoordinates1}. We define the vector fields $\widetilde X_i^p= (F_{p,\bX})^{-1}_*X_i$ and the smooth functions 
$a_{ij}^p\in C^\infty(V_p)$, such that for every $i=1,\dots,\n$ and for every $x\in V_p$, we have
		\[
		\widetilde X_i^p(x)=\sum_{j=1}^\n a_{ij}^p(x)\partial_j.
		\]
		Then, for any $i,j=1,\dots,\n$, there exist a unique polynomial $b^p_{ij}$ and a smooth function $\rho^p_{ij}\in C^\infty(V_p)$ such that 
		$a^p_{ij}=b^p_{ij}+\rho^p_{ij}$ and the following conditions hold.
		\begin{itemize}
			\item[(a)] If $w_j\geq w_i$, we have that
			\begin{enumerate} 
				\item
				$b^p_{ij}$ is $\delta$-homogeneous of degree $w_j-w_i$,
				\item
				 $\lim_{x\to 0} \|x\|^{w_i-w_j}\rho^p_{ij}(x)=0$, in particular $\rho^p_{ij}(0)=0$.
			\end{enumerate}
			\item[(b)] If $w_j=w_i$, then $b^p_{ij}=\delta_{ij}$ (where $\delta_{ij}$ denotes
			the usual Kronecker delta)
		\item[(c)] If $w_j<w_i$, then $b^p_{ij}=0$.
		\end{itemize}
Moreover, if we define for $i=1,\dots, \n$ and $r>0$ the vector fields
\begin{equation}\label{eq:Xcappuccio}
		\widehat X_i^p(x)=\sum_{j=1}^\n b^p_{ij}(x)\partial_j\quad\text{and}\quad
\widetilde X_i^{p,r}=r^{w_i}\,(\delta_{1/r})_*(\widetilde X^p_i),
\end{equation}
then $\widetilde X_i^{p,r}$ converges to $\widehat X^p_i$ as $r\to0$ in the $C_{\rm loc}^\infty$-topology.
In addition, the frame $(\widehat X_1^p,\dots, \widehat X_\m^p)$ defines a stratified group structure on $\R^\n$.
\end{theorem}

\begin{rmk}\label{rmk:nilpotentapp}
	According to Theorem~\ref{th:tangentspace}, the frame 
	$\widehat\bX^p=(\widehat X^p_1,\ldots,\widehat X^p_\n)$ 
	representing the nilpotent approximation has the following property: the polynomials $b_{ij}^q$ are obtained 
	by taking the homogeneous part of degree $w_j-w_i$ in the Taylor's expansion of $a_{ij}^q$.
	As a consequence, the coefficients of these polynomials are expressed in terms of the iterated partial derivatives of 
	$a^p_{ij}$ at the origin. This observation will be important in the proof of Theorem~\ref{thm:unifblowup}.
\end{rmk}

\begin{definition}[Nilpotent approximation and tangent sub-Riemannian distance]\label{def:nilpApprox} \rm
Let $(M,\cD,g)$ be an equiregular sub-Riemannian manifold and let $\bX=(X_1,\dots, X_\n)$ be a privileged frame of smooth vector fields on an open set $W\subset M$. If $p\in W$, the unique frame $\widehat\bX^p=(\widehat X^p_1,\ldots,\widehat X^p_\n)$ in \eqref{eq:Xcappuccio}, provided by Theorem~\ref{th:tangentspace} represents the so-called {\em nilpotent approximation of the frame $(X_1, \dots,X_\n)$ around $p$}.
We denote by $\widehat g_p$ the sub-Riemannian metric of $\R^\n$ that makes $\widehat \bX_h=(\widehat X^p_1,\ldots,\widehat X^p_\m)$ a horizontal and orthonormal frame generating a stratified Lie algebra. 
The previous frame generates a sub-Riemannian distance $\widehat d_p$ defined on all couples of points of $\R^\n$, hence Theorem~\ref{th:tangentspace} implies that this distance is left invariant and 1-homogeneous with respect to the dilations \eqref{df:tangdilations}.
Taking into account that the nilpotent approximation can be seen as Gromov-Hausdorff limit of the rescaled sub-Riemannian manifold, \cite{Mitchell85},
it is natural to call $\widehat d_p$ the {\em tangent sub-Riemannian distance} of $M$ at $p$.
\end{definition}
\begin{definition}[Tangent sub-Riemannian balls]\label{def:NAballs}
We consider the exponential coordinates of the first kind $F_{p,\bX}\colon V_p\to F_{p,\bX}(V_p)$ associated
with a regular point $p\in M$ of a sub-Riemannian manifold $(M,\cD,g)$.
We have denoted by $\bX=(X_1,\ldots,X_\n)$ a privileged frame on a neighborhood of $p$
and $V_p\subset\R^\n$ is an open neighborhood of $0\in\R^\n$. 
The {\em tangent sub-Riemannian ball} is the open metric ball of center $x\in\R^\n$ and radius $r>0$ with respect to 
the tangent sub-Riemannian $\widehat d_p$ is
\[
\widehat B_p(x,r)=\set{y\in \R^n:\widehat d_p(x,y)<r}.
\]
The corresponding closed ball is $\widehat \B_p(x,r)=\{y\in\R^\n: \widehat d_p(x,y)\leq r \}$.
If we wish to emphasize the frame that generate the nilpotent approximation, we may also write $\widehat B_{p,\bX}(x,r)$ and $\widehat \B_{p,\bX}(x,r)$ 
in place of $\widehat B_p(x,r)$ and $\widehat \B_p(x,r)$, respectively.
\end{definition}

\begin{rmk}
The use of the entire frame $\bX$ in the notation for the metric balls $\widehat B_{p,\bX}(x,r)$ and $\widehat \B_{p,\bX}(x,r)$ is justified by the fact that before getting the horizontal frame $\widehat \bX_h$ of the nilpotent approximation, we have to first fix a system of exponential coordinates arising from $\bX$. 
\end{rmk}

\begin{notation}
The open and closed Euclidean balls of center $x$ and radius $r>0$ in $\R^n$ are denoted by $B_\rE(x,r)$ and $\B_\rE(x,r)$, respectively.
\end{notation}

The following resut is a well known metric version of the nilpotent approximation, that can be found for instance in \cite[Lemma~20.20]{ABB20}, \cite[Theorem~3.5]{AGM15BVSubRiem}, \cite[Theorem~2.9]{DonVittone2019FineProp} or \cite[Theorem~2.2]{Jean2014}.

\begin{theorem}[Pointed blow-up] \label{th:localgroup}
	Let us consider an equiregular sub-Riemannian manifold $(M, \mathcal D, g)$ and $p\in M$. Let $\widehat \bX^p_h=(\widehat X_1^p,\dots,\widehat X_\m^p)$ be the horizontal frame defined by \eqref{eq:Xcappuccio}. Then for every $R>0$ we have
	\begin{equation}\label{eq:localgroup}
	\lim_{r\to 0} \left(\sup\left\{\frac{|\widetilde d_p(x,y)-\widehat d_p(x,y)|}{r}: x,y\in \widehat B_p(0,rR)\right\}\right)=0,
	\end{equation}
where $\widetilde{d}_p$ is the local distance  (Section~\ref{subsect:inducedSR})
and $\widehat{d}_p$ is the tangent sub-Riemannian distance (Definition~\ref{def:nilpApprox}).
\end{theorem}

From standard arguments, Theorem~\ref{th:localgroup} yields the following well known result. 

\begin{theorem}\label{thm:BallEquiv}
Let $(M, \mathcal D, g)$ be an equiregular sub-Riemannian manifold and let $p\in M$. 
Then for every $\varepsilon >0$ there exist $R>0$ such that, for every $r\in(0,R)$, we have
\begin{equation}\label{incl:Btildeqreps}
	\widetilde B_p(0,(1-\varepsilon)r)\subseteq\widehat B_p(0,r)\subseteq \widetilde B_p(0,(1+\varepsilon)r),
\end{equation}
where $\widetilde B_q(x,r)$ denotes local sub-Riemannian ball (Definition~\ref{def:balls}) 
and $\widehat B_q(x,r)$ is the tangent sub-Riemannian ball (Definition \ref{def:NAballs}).
\end{theorem}

\begin{rmk}\label{rmk:expidentity}
	If $q\in M$ and $\bX=(X_1,\dots,X_\n)$ is a privileged frame, we consider the nilpotent approximation $\widehat \bX^q=(\widehat X^q_1,\dots,\widehat X^q_\n)$ and the associated exponential map
	\[
	F_{0,\widehat \bX^q}(x_1,\dots, x_\n)=\exp(x_1\widehat X_1^q+\dots+x_\n \widehat X^q_\n)(0).
	\]
	Using basic properties of flows and the convergence of the rescaled vector fields $\widetilde X^{p,r}_i$ in Theorem~\ref{th:tangentspace} one can check that 
	$F_{0,\widehat \bX^q}\colon \R^\n\to \R^\n$ is the identity mapping, see the arguments of \cite[Remark~2.6]{MPV2018}.
\end{rmk}

\subsection{Sets of finite perimeter in sub-Riemannian manifolds}\label{sect:perimeter}

In this section a sub-Riemannian measure manifold $(M,\cD,g,\omega)$ is fixed.
We introduce the family of {\em horizontal subunit vector fields} 
\[
\cD^g=\left\{X\in\cD: g_x(X(x),X(x))\leq 1\,\,\text{for all}\, x\in M\right\}.
\]
The notion of perimeter measure with respect to a volume form $\omega$ requires that also the divergence refers to the same volume form.
We have the general definition 
\beq \label{eq:divomega}
L_X(\omega)=(\div_\omega X)\,\omega, 
\eeq
where $X$ is a vector field and $L_X$ is its associated Lie derivative.

\begin{definition}
	[Sets of finite perimeter]\label{def:SRPer}
	We say that a Borel set $E\subset M$ has {\it finite perimeter} if
	\begin{equation}
		\sup\left\{\int_E \,\div_\omega(\varphi X)\,\omega: X\in\cD^g,\,\varphi\in C^\infty_c(M),\,|\varphi|\leq 1\right\}<\infty.
	\end{equation}
	The supremum is denoted by $\|D_{\omega,g}\bu_E\|(M)$. It is
	the {\em sub-Riemannian perimeter} of $E$.
\end{definition}
It can be proved that the set function 
\[
\|D_{\omega,g}\bu_E\|(U)=\sup\left\{\int_U\,\bu_E\,\div_\omega(\varphi X)\,\omega: X\in\cD^g,\,\varphi\in C^\infty_c(U),\,|\varphi|\leq 1\right\}
\]
defined on all open sets $U$ can be extended to a Radon measure on $M$,
see \cite{AGM15BVSubRiem} for more information.

\section{Spherical factor in equiregular sub-Riemannian manifolds}

The present section introduces the spherical factor in SR manifolds, that is the key geometric function to compute the spherical measure of a hypersurface. 

We start with the definition of spherical factor using exponential coordinates of the first kind arising from a privileged orthonormal frame.

\begin{definition}[Spherical factor]\label{def:beta}
Let $(M,\mathcal D,g)$ be a sub-Riemannian manifold and let $p\in M$ be a regular point. We denote by the same symbol $g$ a Riemannian metric that extends the sub-Riemannian metric.
	Let $(X_1,\dots,X_\n)$ be a privileged orthonormal frame in a neighborhood of $p$.
	Let $\nu\in\cD_p\sm\set{0}$ and consider its orthogonal subspace $\Pi(\nu)\subset T_pM$ with respect to $g$.
	We denote by $(\widehat X_1^p,\dots, \widehat X_\n^p)$ the nilpotent approximation of $(X_1,\dots, X_\n)$ at $p$
	and consider the exponential coordinates of the first kind $F_{p,\bX}$ centered at $p$, given by \eqref{eqdef:exponentialcoordinates1}. The {\em spherical factor at $p$ with respect to $\nu$} is the number
	\begin{equation}\label{eq:sphericalfactor}
		\beta_{d,g}(\nu)=
		\max_{z\in\widehat B_p(0,1)}\mathcal H^{\n-1}_\rE((\dd F_{p,\bX})(0)^{-1}(\Pi(\nu))\cap \widehat \B_{p,\bX}(z,1)),	
	\end{equation}
	where $\mathcal H_\rE^{\n-1}$ is the $(\n-1)$-dimensional Hausdorff measure with respect to the Euclidean distance $\R^\n$, see Definition~\ref{d:sizephi}.
	The closed metric unit ball $\widehat \B_{p,\bX}(z,1)$ refers to the SR distance 
	$\widehat{d}_p$ associated with the horizontal orthonormal frame $\widehat \bX^p_h=(\widehat{X}^p_1,\ldots,\widehat{X}^p_\m)$.
\end{definition}

A priori, the previous definition may depend on the system of exponential coordinates of the first kind that we have chosen. We use the Euclidean Hausdorff measure that refers to these coordinates. The next theorem is the key result to prove that the definition \eqref{eq:sphericalfactor} is well posed.

\begin{theorem}[Change of exponential coordinates of the first kind]\label{th:intrinsicfederer}
	Let $(M,\mathcal D,g,\omega)$ be a sub-Riemannian measure manifold and denote by the same symbol $g$ a Riemannian metric on $M$ that extends the sub-Riemannian metric. We assume that $p\in M$ is a regular point and consider two privileged orthonormal frames $\bX=(X_1,\dots, X_\n)$ and $\bY=(Y_1,\dots, Y_\n)$ in an open neighborhood $W$ of $p$. 
	According to \eqref{eqdef:exponentialcoordinates1}, we introduce the exponential coordinates of the first kind  $F_{p,\bX}, F_{p,\bY}\colon V\to W$ associated with $\bX$ and $\bY$ respectively, around $p$. The set $V\subset\R^\n$ is an open neighborhood of $0\in\R^\n$.
	The frames $\widehat\bX^p=(\widehat X_1^p,\dots, \widehat X^p_\n)$ and $\widehat\bY^p=(\widehat Y_1^p,\dots, \widehat Y_\n^p)$ denote the nilpotent approximations of $\bX$ and $\bY$ at $p$, see Definition~\ref{def:nilpApprox}. 
	Then the following facts hold:
	\begin{itemize}
		\item[(i)] the family of maps $\delta_{1/\varepsilon}\circ F_{p,\bY}^{-1}\circ F_{p,\bX}\circ \delta_\varepsilon\colon V\to \mathbb R^\n$ uniformly converges to the restriction of a linear Euclidean isometry $\widehat L\colon \R^\n\to\R^\n$ as $\varepsilon \to 0$;
		\item[(ii)] we have $\widehat L=\dd(F_{p,\bY}^{-1}\circ F_{p,\bX})(0)$ and the matrix associated with $\widehat L$ is block diagonal;
		\item[(iii)] if we denote by $\widehat d_{p,\bX}$ and $\widehat d_{p,\bY}$ the sub-Riemannian distances associated with $\widehat\bX^p$ and $\widehat\bY^p$,
		respectively, then $\widehat d_{p,\bY}(\widehat L(x),\widehat L(y))=\widehat d_{p,\bX}(x,y)$ for every $x,y\in \mathbb R^\n$.
	\end{itemize}
\end{theorem}

\begin{proof}
	Since $\bX$ and $\bY$ are both orthonormal privileged frames, there exists a smooth matrix-valued map $C\colon W\to O(\n,\R)$ such that $C(q)=(c_j^i(q))_{ij}$ is a real orthogonal matrix and
	\beq\label{eq:Y_jc^i_j}
	Y_j(q)=\sum_{i=1}^{\n}c_j^i(q)X_i(q)
	\eeq
	for every $q\in W$ and $j=1,\dots,\n$. More precisely, the previous sum can be also written using weights
	\beq\label{eq:Y_jc^i_jw_i}
	Y_j=\sum_{i:w_i=w_j}c_j^iX_i.
	\eeq
	Thus $C(q)$ is a block diagonal matrix and so is $C(q)^{-1}$ for every $q\in U$. Precisely, we have
	\beq\label{eq:X_jc^i_jw_i}
	X_i=\sum_{j=1}^\n(C^{-1})_i^j\,Y_j=\sum_{j:w_j=w_i}(C^{-1})_i^j\,Y_j.
	\eeq
	We introduce the vector fields $\widetilde X_i^p=(F_{p,\bX}^{-1})_*X_i$ and $\widetilde Y_i^p=(F_{p,\bY}^{-1})_*Y_i$ for all $i=1,\ldots,\n$. Up to shrinking $V$, we may assume that both $\widetilde X^p_i$ and $\widetilde Y^p_i$ are well defined on $V$. 
	By definition of $F_{p,\bX}$ and $F_{p,\bY}$, we also have $\widetilde X^p_i(0)=\widetilde Y^p_i(0)=e_i\in\R^\n$ for every $i=1,\dots, \n$. As usual, $(e_1,\ldots,e_\n) $ denotes the canonical basis of $\R^\n$. The differential $\dd(F_{p,\bY}^{-1}\circ F_{p,\bX})(0)$ is related to the matrix $C(p)$, indeed we have
	\begin{align}
		\dd(F_{p,\bY}^{-1}\circ F_{p,\bX})(0)(e_i)&=\dd F_{p,\bY}^{-1}(p)(X_i(p))=\sum_{j=1}^\n(C^{-1}(p))^j_i \dd F_{p,\bY}^{-1}(p)(Y_j(p)) \nonumber\\
		&=\sum_{j:w_j=w_i}(C^{-1}(p))^j_ie_j. \label{eq:dG_p-1F_p}
	\end{align}
	The second equality follows from \eqref{eq:X_jc^i_jw_i}.
	Being $C(p)$ orthogonal, we notice that 
	\[
	(C^{-1}(p))^j_i=(C(p)^T)^j_i=c^i_j(p).
	\]	
	We define the smooth functions $\tilde a_\ell^s\colon V\to\R$ such that
	\beq\label{eq:a_s^l}
	(F_{p,\bY}^{-1}\circ F_{p,\bX})_*(\widetilde X^p_\ell)=\sum_{s=1}^{\n}\tilde a_\ell^s \widetilde X^p_s.
	\eeq
	Applying $(F_{p,\bY}^{-1})_*$ to equality \eqref{eq:Y_jc^i_jw_i}, we get
	\begin{align*}
		\widetilde Y^p_j&=(F_{p,\bY}^{-1})_*\pa{\sum_{i:w_i=w_j}c^i_j X_i}=
		\sum_{i:w_i=w_j}c^i_j\circ F_{p,\bY}\, (F_{p,\bY}^{-1})_*((F_{p,\bX})_*(\widetilde X^p_i)) \\
		&=\sum_{\substack{i:w_i=w_j\\1\le s\le\n}}c^i_j\circ F_{p,\bY}\,\tilde a_i^s\widetilde X^p_s  
	\end{align*}
	where the last equality is a consequence of \eqref{eq:a_s^l}.
	Evaluating the previous equalities at $0$, we have proved that
	\[
	\sum_{i=1}^\n c^i_j(p)\,\tilde a_i^s(0)=\sum_{i:w_i=w_j}c^i_j(p)\,\tilde a_i^s(0)=\delta_j^s,
	\]
	where $\delta_j^s$ is the Kronecker delta. In other words, the matrix $A=(\tilde{a}^s_i(0))_{si}$ satisfies $A=C(p)^{-1}$.
	We define the smooth functions 
	\beq\label{eqdef:sigmaij}
	\sigma^s_j=\sum_{i:w_i=w_j}c^i_j\circ F_{p,\bY}\,\tilde a_i^s
	\eeq
	and observe that 
	\[
	\widetilde Y^p_j=\sum_{s=1}^{\n}\sigma^s_j\widetilde X^p_s.
	\]
	Using dilations $\delta_r$, it follows that
	\beq\label{eq:Y_j{p,r}}
	 r^{w_j}(\delta_{1/r})_*\widetilde Y^p_j=\sum_{s=1}^{\n}\sigma^s_j\circ\delta_r\,r^{w_j-w_s}[r^{w_s}(\delta_{1/r})_*\widetilde X^p_s].
	\eeq
	Theorem~\ref{th:tangentspace} implies that the ordered sets	
	\[
	\widetilde\bX^{p,r}=(\widetilde X^{p,r}_1,\dots, X^{p,r}_\n)\quad\text{and}\quad \widetilde\bY^{p,r}=(\widetilde Y^{p,r}_1,\dots, Y^{p,r}_\n)
	\]
	with $\widetilde X^{p,r}_i=r^{w_i}(\delta_{1/r})_*\widetilde X^p_i$ and $\widetilde Y^{p,r}_i=r^{w_i}(\delta_{1/r})_*\widetilde Y^p_i$ are frames for sufficiently small $r>0$ and uniformly converge to 
	$\widehat\bX^p=(\widehat X_1^p,\dots, \widehat X^p_\n)$ and $\widehat\bY^p=(\widehat Y_1^p,\dots, \widehat Y_\n^p)$ on compact sets, respectively, as $r\to0$. We may define the ``moving dual basis'' $\eta^{p,r}_j\colon \R^\n\to\R$ such that
	\beq
	\eta^{p,r}_j(\widetilde X^{p,r}_i)=\delta^j_i
	\eeq
	for all $i,j=1,\ldots,\n$. It uniformly converges to $\eta^p_j$ on compact sets, as $r\to0$, where
	\beq
	\eta^{p}_j(\widehat X^{p}_i)=\delta^j_i.
	\eeq
	Applying the moving dual basis to \eqref{eq:Y_j{p,r}}, we get
	\[
	\eta^{p,r}_i(\widetilde Y^{p,r}_j)=\sum_{s=1}^{\n}\sigma^s_j\circ\delta_r\,r^{w_j-w_s}\eta^{p,r}_i(\widetilde X^{p,r}_s)
	=\sigma^i_j\circ\delta_r\,r^{w_j-w_i},
	\]
	therefore for every $i,j=1,\ldots,\n$ there exists
	\beq\label{eq:intrinsicDiff}
	\lim_{r\to0} \frac{\sigma^i_j\circ\delta_r}{r^{w_i-w_j}}=\eta^p_i(\widehat Y^p_j).
	\eeq
	The previous limits can be read as a kind of intrinsic differentiability result. Combining the previous limit and formula \eqref{eqdef:sigmaij}, it follows that
	\begin{align}
	\lim_{r\to0} \frac{\tilde a^s_i\circ\delta_r}{r^{w_s-w_i}}&=\lim_{r\to0} \frac{\sum_{j:w_j=w_i}(C^{-1})^j_i\circ F_{p,\bY}\circ\delta_r\;\sigma_j^s\circ\delta_r}{r^{w_s-w_i}}\nonumber \\
	&=\sum_{j:w_j=w_i}(C^{-1})^j_i(p)\,\lim_{r\to0}\frac{\sigma_j^s\circ\delta_r}{r^{w_s-w_j}}\nonumber \\
	&=\sum_{j:w_j=w_i}(C^{-1})^j_i(p)\,\eta^p_s(\widehat Y^p_j). \label{eq:limita_i^s}
	\end{align}
	We now define the family of mappings
	\[
	H_\ep=\delta_{1/\varepsilon}\circ F_{p,\bY}^{-1}\circ F_{p,\bX}\circ \delta_\varepsilon,
	\]
	observing that 
	\begin{align*}
	H_\ep(x)&=\delta_{1/\ep}\left(F_{p,\bY}^{-1}\Big(\exp\Big(\sum_{i=1}^{\n}\ep^{w_i}x_i\,(F_{p,\bX})_*\widetilde X_i\Big)(p)\Big)\right) \\
	&=\exp\Big(\sum_{i=1}^{\n}\ep^{w_i}x_i\,(\delta_{1/\ep})_*(F_{p,\bY}^{-1}\circ F_{p,\bX})_*\widetilde X_i\Big)\circ (\delta_{1/\ep}\circ F_{p,\bY}^{-1})(p) \\
	&=\exp\Big(\sum_{i=1}^{\n}\ep^{w_i}x_i\big(\sum_{s=1}^{\n}\tilde a_i^s\circ\delta_\ep\, (\delta_{1/\ep})_*\widetilde X^p_s\big)\Big)(0),
	\end{align*}
where the last equality follows from \eqref{eq:a_s^l}. From the previous definitions, we obtain
\begin{align*}
	H_\ep(x)&=\exp\Big(\sum_{i=1}^{\n}\ep^{w_i}x_i\big(\sum_{s=1}^{\n}\tilde a_i^s\circ\delta_\ep\, (\delta_{1/\ep})_*\widetilde X^p_s\big)\Big)(0) \\
	&=\exp\Big(\sum_{i=1}^{\n}x_i\big(\sum_{s=1}^{\n}\frac{\tilde a_i^s\circ\delta_\ep}{\ep^{w_s-w_i}}\, \widetilde X^{p,\ep}_s\big)\Big)(0).
\end{align*}
Due to the limit \eqref{eq:limita_i^s} and the uniform convergence of $\widetilde X^p_s$, it follows that 
\begin{align*}
\lim_{\ep\to0}H_\ep(x)&=\exp\Big(\sum_{i=1}^\n x_i\sum_{s=1}^\n \sum_{j:w_j=w_i}(C^{-1}(p))^j_i\,\eta^p_s(\widehat Y^p_j) \widehat X^p_s\Big)(0) \\
&=\exp\Big(\sum_{j=1}^\n\sum_{i=1}^\n(C^{-1}(p))^j_ix_i\,\widehat Y^p_j\Big)(0).
\end{align*}
Taking into account Remark~\ref{rmk:expidentity}, we have proved that
\begin{equation}\label{eq:limH_ep}
	\lim_{\ep\to0}H_\ep(x)=C^{-1}(p)x\in\R^\n.
\end{equation}
We set the linear mapping $\widehat L(x)=C(p)^{-1}x$ for every $x\in\R^\n$, that is a Euclidean isometry, hence the first point of our claim is proved. 
The second claim follows combining \eqref{eq:dG_p-1F_p} and \eqref{eq:limH_ep}.

For the third claim, we arbitrarily fix $R>0$ and notice that \eqref{eq:localgroup} gives 
\[
\lim_{\ep \to0}\sup_{x,y\in\widehat B(0,R)}\left|\frac{d(F_{p,\bX}(\delta_\ep x),F_{p,\bX}(\delta_\ep y))}{\ep}-\widehat d_{p,\bX}(x,y)\right|=0,
\]
where $\widehat d_{p,\bX}$ is the sub-Riemannian distance associated with $\widehat \bX^p$
and $\widehat B_{p,\bX}(0,R)$ is the metric ball with respect to $\widehat d_{p,\bX}$.
Denoting by $o(1)$ any infinitesimal function as $\ep\to0$, we get
\begin{align*}
		\sup_{x,y\in  \widehat B_{p,\bX}(0,R)}|\widehat d_{p,\bY}(&H_\varepsilon(x),H_\varepsilon(y))-\widehat d_{p,\bX}(x,y)|\leq o(1) \\
		&+\sup_{x,y\in \widehat B_{p,\bX}(0,R)}\left|\widehat d_{p,\bY}(H_\varepsilon(x),H_\varepsilon(y))-\frac{d(F_{p,\bX}(\delta_\varepsilon x),F_{p,\bX}(\delta_\varepsilon y))}{\ep}\right|.
\end{align*}
The second addend of the previous inequality can be written as 
\begin{align}
\sup_{\xi,\xi'\in H_\ep(\widehat B_{p,\bX}(0,R))}\left|\widehat d_{p,\bY}(\xi,\xi')-\frac{d(F_{p,\bY}(\delta_\ep\xi),F_{p,\bY}(\delta_\ep\xi'))}{\ep}\right|.
\end{align}
The uniform convergence of $H_\ep$ proved in the previous step yields $R_1>R$ such that
\[
H_\ep(\widehat B_{p,\bX}(0,R))\subset \widehat B_{p,\bY}(0,R_1)
\]
for $\ep>0$ sufficiently small. The metric ball $\widehat B_{p,\bY}(0,R_1)$ is defined by $\widehat d_{p,\bY}$. We get
\begin{align*}
	\sup_{x,y\in  \widehat B_{p,\bX}(0,R)}|\widehat d_{p,\bY}(H_\varepsilon(x)&,H_\varepsilon(y))-\widehat d_{p,\bX}(x,y)|\leq o(1) \\
	&+\sup_{\xi,\xi'\in \widehat B_{p,\bY}(0,R_1)}\left|\widehat d_{p,\bY}(\xi,\xi')-\frac{d(F_{p,\bY}(\delta_\varepsilon \xi),F_{p,\bY}(\delta_\varepsilon \xi'))}{\ep}\right|.
\end{align*}
By the limit \eqref{eq:localgroup}, we get
\begin{equation}
	\lim_{\ep\to0}\sup_{x,y\in  \widehat B_{p,\bX}(0,R)}|\widehat d_{p,\bY}(H_\varepsilon(x),H_\varepsilon(y))-\widehat d_{p,\bX}(x,y)|=0,
\end{equation}
that immediately proves our last claim, therefore concluding the proof.
\end{proof}

\begin{corollary}
Let $(M,\mathcal D,g,\omega)$ be a sub-Riemannian measure manifold and denote by the same symbol $g$ a Riemannian metric on $M$ that extends the sub-Riemannian metric. We consider two privileged orthonormal frames $\bX=(X_1,\dots, X_\n)$ and $\bY=(Y_1,\dots, Y_\n)$ in an open neighborhood $W$ of $p\in M$ and denote by $F_{p,\bX}, F_{p,\bY}\colon V\to  W$ the exponential coordinates of the first kind associated with $\bX$ and $\bY$ around $p$, respectively, see \eqref{eqdef:exponentialcoordinates1}. 
Let $\nu\in\cD_p\sm\set{0}$ and denote by $\Pi(\nu)\subset T_pM$ its orthogonal subspace with respect to $g$. We consider the nilpotent approximations $\widehat \bX^p=(\widehat X_1^p,\dots, \widehat X_\n^p)$ and $\widehat \bY^p=(\widehat Y_1^p,\dots, \widehat Y_\n^p)$  of $\bX$ and $\bY$ at $p$.
If $\widehat L\colon \R^\n\to\R^\n$ is the Euclidean isometry of Theorem~\ref{th:intrinsicfederer}, it holds
	\begin{equation}\label{eq:=sphericalfactor}
	\mathcal H^{\n-1}_\rE((\dd F_{p,\bX})(0)^{-1}(\Pi(\nu))\cap \widehat \B_{p,\bX}(z,1))=\mathcal H^{\n-1}_\rE((\dd F_{p,\bY})(0)^{-1}(\Pi(\nu))\cap \widehat \B_{p,\bY}(\widehat L(z),1)),	
	\end{equation}
	where $\mathcal H_\rE^{\n-1}$ denotes the standard Euclidean $(\n-1)$-dimensional Hausdorff measure of $\R^\n$ and $z\in \R^n$.
\end{corollary}

\begin{proof}
Since $\widehat{L}$ is an Euclidean isometry, it follows that  
\[
\mathcal H^{\n-1}_\rE((\dd F_{p,\bX})(0)^{-1}(\Pi(\nu))\cap \widehat \B_{p,\bX}(z,1))
\]
is equal to 
\[
\mathcal H^{\n-1}_\rE\Big(\widehat L\big((\dd F_{p,\bX})(0)^{-1}(\Pi(\nu))\big)\cap \widehat L\big(\widehat \B_{p,\bX}(z,1)\big)\Big).
\] 
From claim (ii) of Theorem~\ref{th:intrinsicfederer}, we get
\beq\label{eq:LhatFG}
\widehat L\big((\dd F_{p,\bX})(0)^{-1}(\Pi(\nu))\big)=(\dd F_{p,\bY})(0)^{-1}(\Pi(\nu)).
\eeq
Claim (iii) of Theorem~\ref{th:intrinsicfederer} yields $\widehat L\big(\widehat \B_{p,\bX}(z,1)\big)=\widehat \B_{p,\bY}(L(z),1)$. Thus, using the first equality and \eqref{eq:LhatFG}, the proof of \eqref{eq:=sphericalfactor} is complete.
\end{proof}

\begin{corollary}\label{cor:betaintrinsic}
Let $(M,\mathcal D,g)$ be a sub-Riemannian manifold with a regular point $p\in M$ and a Riemannian metric $g$ that extends the sub-Riemannian metric. If $d$ is the associated sub-Riemannian distance and $\nu\in\cD_p\sm \set{0}$, then the spherical factor $\beta_{d,g}(\nu)$ is independent of the choice of the exponential coordinates of the first kind.
\end{corollary}
\begin{proof}
We consider two privileged orthonormal frames $\bX=(X_1,\dots, X_\n)$ and $\bY=(Y_1,\dots, Y_\n)$ in an open neighborhood $W$ of $p\in M$. Following \eqref{eqdef:exponentialcoordinates1}, we denote by $F_{p,\bX}, F_{p,\bY}\colon V_q\to  W_q$ the exponential coordinates of the first kind associated with $\bX$ and $\bY$ at $p$. The open sets $V_q\subset\R^\n$ and $W_q\subset M$ are neighborhoods of $0$ and $q$, respectively. 
We consider the Euclidean isometry $\widehat{L}$ of Theorem~\ref{th:intrinsicfederer}, hence combining \eqref{eq:=sphericalfactor} and the definition of spherical factor \eqref{eq:sphericalfactor} our claim is established.
\end{proof}

\section{Uniform nilpotent approximation}\label{sec:blowup}

In this section we prove the uniform convergence of the rescaled sub-Riemannian distances to the distance of the nilpotent approximation, where the ``blow-up point'' varies in a compact set. 
Theorem~\ref{th:DVGen} is the key tool to establish the {\em uniform nilpotent approximation} (Theorem~\ref{thm:unifblowup}).

In the proof of the next theorem both Definition~\ref{def:SRdX_j} 
and Remark~\ref{rem:equivSRX_j} are considered.

\begin{theorem}[Uniform convergence of SR distances]\label{th:DVGen} Let us consider $\m\leq\n$, a compact metric space $\Xi$, two frames $\bX^{j,q}_h=(X_1^{j,q},\dots, X_\m^{j,q})$ and $\bX^q_h=(X_1^q,\dots, X_\m^q)$ on the Euclidean ball $B_\rE(x_0,R_0)\subset\R^\n$ satisfying the Chow's condition for each $q\in \Xi$ and $j\in\N$.
We assume that all partial derivatives 
	$$
	B_{\rm E}(x_0,R_0)\times \Xi\ni (x,q)\mapsto \partial_x^\alpha X_i^q(x)\quad \text{and}\quad B_{\rm E}(x_0,R_0)\times \Xi\ni (x,q)\mapsto \partial _x^\alpha X_i^{j,q}(x)
	$$
	are continuous for every $i=1,\dots, \m$, $j\in \mathbb N$ and every multi-index 
	$\alpha\in\N^\n$. 
	We also suppose that for each $i=1,\dots,\m$ the vector field $X_i^{j,q}$ converges to $X_i^q$ in 
	$B_{\rm E}(x_0,R_0)\subset\R^\n$ with respect to the $C_{\rm loc}^\infty$-topology as $j\to \infty$, uniformly with respect to $q\in \Xi$.

Then there exists $0<r_0<R_0$, such that the SR distance $d_q^j$ associated with $\bX^{j,q}_h$ on $B_\rE(x_0,R_0)$ converges to the SR distance $d_q$ associated with $\bX^q_h$ on $B_\rE(x_0,R_0)$ in the topology of $L^\infty(B_{\rm E}(x_0,r_0)\times B_{\rm E}(x_0,r_0))$ as $j\to \infty$, uniformly with respect to $q\in \Xi$.
\end{theorem}
\begin{proof}
We first notice that the step $s^q(x)\geq 1$ of the frame $\bX_h^q$ at $x\in \B_\rE(x_0,R_0')$ is upper semicontinuous with respect to $(q,x)\in \Xi\times \B_\rE(x_0,R_0')$ for any fixed $0<R_0'<R_0$, hence it is bounded on the compact set $\Xi\times \B_\rE(x_0,R_0')$. Thus, up to slightly reducing $R_0$, the compactness of $\Xi$ allows us to consider $s\in\N$ such that $s^q(x)\leq s$ for every $x\in B_{\rm E}(x_0,R_0)$ and $q\in \Xi$. Since each $X_i^{j,q}$ is converging to $X_i^q$ in the $C^\infty_{\rm loc}$-topology and uniformly as $q$ varies in $ \Xi$, we can choose a sufficiently large $\widetilde J\in \mathbb N$ such that $s_j^q(x)\leq s$ for every $j\geq \widetilde J$, $q\in \Xi$ and $x\in B_{\rm E}(x_0,R_0)$. 
Up to ignoring a finite number of terms in the sequence of frames and relabeling the indexes, we may assume that the uniform bound on the step holds for every $j\in \mathbb N$.

Due to \cite[Claim~3.3]{BBP} joined with \cite[Proposition~5.8]{BBP} (see also \cite[Proposition~1.1 and Theorem~4]{NagelSteinWainger85}), for a smooth
family of vector fields $\bY_h=(Y_1,\ldots,Y_\m)$ on an open set $\Omega$, that satisfy Chow's condition,
for any compact set $K\subset \Omega$, there exists a constant $C>0$ such that 
\begin{equation}\label{eq:C_KdjqY}
	\frac1{C}|x-y|\le d_{\bY_h,\Omega}(x,y)\le C|x-y|^{1/s}
\end{equation}
for $x,y\in K$, where $d_{\bY_h,\Omega}$ is the SR distance considering admissible curves contained in $\Omega$. According to \cite[Claim~3.3]{BBP}, the key point of this result is that the constant $C>0$ only depends on 
$K$, the dimension $\n$ of the space, the rank $\m$, the step $s$ (depending on $K$), the uniform upper bound on the $C^\ell(K)$-norms of the vector fields $Y_i$ for some large enough $\ell\in\N$ (depending on $s$) and the positive lower bound of
\[
\inf_{x\in K}\max_{|I_i|\leq s} \det ( Y_{I_1}(x)| \dots|Y_{I_\n}(x)),
\]
where $Y_{I_j}$ denotes the iterated commutators with respect to the multi-index $I_j$. 
As a consequence of our assumptions,
and in particular from the uniform convergence of $X_i^{j,q}$ to $X_i^q$ in the $C^\infty_{\rm loc}$-topology as $q$ varies in $\Xi$ and $i=1,\ldots,\m$,
considering $\Omega=B_{\rE}(x_0,R_0)$ and the compact set $K= \B_{\rE}(x_0,\kappa R_0)$ with a fixed $0<\kappa<1$, we can find $C_\kappa\ge1$, depending on $\kappa$ and $\B_\rE(x_0,R_0)$, and $J_0\in\N$ such that for every $q\in \Xi$ and $j\ge J_0$,  the following inequalities hold
\begin{equation}\label{eq:C_Kdjq}
	\frac1{C_\kappa}|x-y|\le d_q(x,y)\le C_\kappa|x-y|^{1/s}\quad\text{and}\quad  \frac1{C_\kappa}|x-y|\le d_{q}^j(x,y)\le C_\kappa|x-y|^{1/s}
\end{equation}
where $x,y\in \B_{\rE}(x_0,\kappa R_0)$. We recall that $d_q^j$ is the sub-Riemannian distance 
associated with the frame $\bX^{j,q}_h$ on $B_\rE(x_0,R_0)$ and $d_q$ is the sub-Riemannian distance 
associated with the frame $\bX^q_h$ on $B_\rE(x_0,R_0)$, according to Definition~\ref{def:SRdX_j}.
Due to \cite[Theorem~4]{NagelSteinWainger85}, it is not difficult to notice that the distance used in \cite{BBP} 
is equivalent to our $d_q$ and $d_q^j$, up to a geometric constant.
Notice that $C_\kappa$ also depends on $R_0$ and can be arbitrarily large as $\kappa$ becomes close to 1. We have denoted by 
$\B_{\rE}(x_0,\kappa R_0)$ the closed Euclidean ball
of center $x_0$ and radius $\kappa R_0$. 

If $T>0$, $x\in B_\rE(x_0, R_0)$, $h\in L^\infty([0,T],\R^\m)$, $j\in \N$ and $q\in \Xi$, it is convenient to define  $\gamma_{h,x}^{j,q},\gamma_{h,x}^q\colon [0,T]\to \R^\n$ as the absolutely continuous curves such that $\gamma_{h,x}^{j,q}(0)=\gamma_{h,x}^q(0)=x$ that almost everywhere on $ [0,T]$ satisfy
\beq\label{eq:Cauchykappa0}
\dot \gamma_{h,x}^{j,q}=\sum_{i=1}^\m h_i\, X_i^{j,q}\circ \gamma_{h,x}^{j,q},\qquad\dot \gamma_{h,x}^q=\sum_{i=1}^\m h_i\, X_i^{q}\circ \gamma_{h,x}^q.
\eeq
We divide the proof into several steps.
The next step can be seen as ``uniform version'' of \cite[Lemma 3.2]{DonVittone2019Comp} with respect to $q\in \Xi$, where another difference is that the vector fields are only defined on the open ball $B_\rE(x_0,R_0)$. 

 {\bf Step 1.} {\it Let us consider $J_0\in\N$ such that the second estimate of \eqref{eq:C_Kdjq} holds for all $j\ge J_0$. For every $\kappa_0\in (0,\kappa)$, there exists $T_0>0$ such that the curves $\gamma_{h,x}^{j,q}, \gamma_{h,x}^q$ satisfying \eqref{eq:Cauchykappa0} are well defined in the interval $[0,T_0]$ and are contained in $\B_\rE (x_0,\kappa R_0)$, for every $x\in \B_\rE(x_0,\kappa_0R_0)$, $h\in L^\infty([0,T_0],\R^\m)$, $\|h\|_\infty\leq 1$, $q\in \Xi$ and $j\geq J_0$. Precisely, we can choose $T_0=(\kappa-\kappa_0)R_0/(2 C_\kappa)$. }

Let $0<\kappa_0<\kappa$ and consider the Euclidean distance between $\B_{\rE}(x_0,\kappa_0 R_0)$ and $\R^\n\sm B_\rE(x_0,\kappa R_0)$, that is $(\kappa-\kappa_0)R_0$.
Let $h\in L^\infty([0,T],\R^\m)$ be such that $\|h\|_\infty \le 1$, 
\[
0<T\le (\kappa-\kappa_0)R_0/(2 C_\kappa)
\]
and take any curve $\gamma_{h,x}^{j,q}$.

We notice that whenever $t\ge0$ is sufficiently small,
such that $\gamma_{h,x}^{j,q}(t)\in \B_\rE(x_0,\kappa R_0)$, using \eqref{eq:C_Kdjq} for all $j\ge J_0$, 
it follows that 
\[
|\gamma_{h,x}^{j,q}(t)-x|\le C_\kappa\, d_q^j\big(\gamma_{h,x}^{j,q}(t),x\big)
\le C_\kappa t \le \frac{(\kappa-\kappa_0)R_0}2.
\]
Thus, the curve $\gamma^{j,q}_{h,x}$ can be extended to all the interval $[0,T]$ and 
\[
\gamma_{h,x}^{j,q}([0,T])\subset \B_\rE(x_0,(\kappa+\kappa_0)R_0/2)
\]
for $j\ge J_0$ and $q\in\Xi$. 
The same inclusion and the analogous estimates hold for the curves $\gamma_{h,x}^q$.
The proof of Step 1 is concluded by setting $T_0=(\kappa-\kappa_0)R_0/(2 C_\kappa)$.

The following step is a suitable version of 
\cite[Lemma 3.3]{DonVittone2019Comp}, again
adapted to our setting.

 {\bf Step 2.} {\it For every $\kappa_0\in (0,\kappa)$ and $\ep>0$, there exists $J_1=J_1(\ep, \kappa, \kappa_0)\in \N$ such that 
\beq\label{eq:stimejepsilon}
	|\gamma_{h,x}^{j,q}(t)-\gamma^q_{h,x}(t)|\leq \ep
\eeq
	for every $j\ge J_1$, $x\in \B_\rE(x_0,\kappa_0R_0)$, $h\in L^\infty([0,T_0],\R^\m)$ with $\|h\|_\infty\le 1$ and $t\in[0,T_0]$.
The number $T_0=(\kappa-\kappa_0)R_0/(2C_\kappa)$ appeared in the proof of Step 1.}

Take $\ep>0$ and consider $J_1=J_1(\ep, \kappa, \kappa_0)\in \N$ such that $J_1\geq J_0$ and
\[
T_0\left(\sum_{i=1}^\m\sup_{x\in \B_\rE(0,\kappa R_0)}|X_i^{j,q}-X_i^q|\right)e^{\m LT_0}\leq \ep
\]
for every $j\geq J_1$, where $L>0$ is an upper bound for the Lipschitz constants of $X_i^{j,q}$ and of $X_i^q$ on $\B_\rE(x_0,\kappa R_0)$, which is uniform as $j\in \N$, $q\in \Xi$, $i=1,\ldots,\m$ and $T_0=(\kappa-\kappa_0)R_0/(2C_\kappa)$.
The existence of $L>0$ with the previous properties follows from the 
convergence of $X_i^{j,q}$ to $X_i^q$ in the $C^\infty_{\rm loc}$-topology,
which is also uniform with respect to $q\in \Xi$. This proves the existence of $J_1$ with 
the above properties.
Since we know that for $0<T\leq T_0$, the curves $\gamma^{j,q}_{h,x}([0,T]),\gamma^{q}_{h,x}([0,T])$ are contained in $\B(x_0,\kappa R_0)$, arguing as in the proof of 
\cite[Lemma~ 3.3]{DonVittone2019Comp}, for $0\le t\le T$, we obtain
\[
|\gamma_{h,x}^{j,q}(t)-\gamma^q_{h,x}(t)|
\le T \pa{\sum_{i=1}^\m \max_{\B_\rE(x_0,\kappa R_0)} |X_i^{j,q}-X_i^q|} e^{\m LT}\le \ep
\]
for all $t\in[0,T]$,  with $0<T\le T_0=R_0(\kappa-\kappa_0)/(2C_\kappa)$ and any $j\ge J_1$ and $x\in\B_\rE(x_0,\kappa_0R_0)$.

{\bf Step 3.} {\it There exists $\kappa_1\in (0,\kappa)$ such that for any $j\geq J_0$ and $q\in \Xi$, all sub-Riemannian geodesics with respect to either the distance $d_q^j$ or $d_q$ connecting points of $\B_\rE(x_0,\kappa_1 R_0)$ are entirely contained in $B_\rE(x_0,\kappa R_0)$.}

By the estimates of \eqref{eq:C_Kdjq}, we first observe that
\beq\label{eq:Pedj}
\max\set{d_q^j(x,y),d_q(x,y)}\leq  C_\kappa (2\kappa R_0)^{1/s}
\eeq
for $x,y\in \B_\rE(x_0,\kappa R_0)$,
 $q\in \Xi$ and $j\ge J_0$. 
For any $\ep>0$ arbitrarily fixed, we consider any absolutely continuous curve $\gamma\colon[0,T]\to B_\rE(x_0,R_0)$ such that
$\gamma(0)=x$, $\gamma(T)=y$, where
either 
\[
\dot\gamma(t)=\sum_{i=1}^{\m} h_i(t) X_i^{j,q}(\gamma(t))\quad\text{or} \quad \dot\gamma(t)=\sum_{i=1}^{\m} h_i(t) X_i^{q}(\gamma(t)),
\]
with $\|h\|_\infty \le 1$ and 
\[
0<T< C_\kappa (2\kappa_1R_0)^{1/s}+\ep.
\]
Notice that by definition of SR distance on $B_\rE(x_0,R_0)$ and the estimates \eqref{eq:Pedj} such curves always exist. 
We apply Step 1 with $\kappa_1$ in place of $\kappa_0$.
Thus, if we knew that for $0<\kappa_1<\kappa$ 
the inequality $T\le T_0= (\kappa-\kappa_1)R_0/(2C_\kappa)$ holds,
then we would have
\beq\label{eq:inclusionB}
\gamma([0,T])\subset \B_\rE(x_0,\kappa R_0).
\eeq
The previous estimate on $T$ actually holds taking $\kappa_1>0$ such that
\beq\label{eq:Ckappa1}
 C_\kappa (2\kappa_1R_0)^{1/s}+\ep\le (\kappa-\kappa_1)R_0/(2C_\kappa)
\eeq
for an arbitrary $0<\ep <\kappa R_0/(2C_\kappa)$.
We fix $\ep =\kappa R_0/(4C_\kappa)$, so that estimate \eqref{eq:Ckappa1} is satisfied whenever
\beq\label{eq:stimakappa1}
\max\set{C_\kappa 2^{1/s},\frac1{2C_\kappa}}\big((\kappa_1R_0)^{1/s}+\kappa_1R_0\big)\le \frac{\kappa R_0}{4C_\kappa}
\eeq
In particular, by the standard Ascoli--Arzelà compactness argument,
we have proved that every SR geodesic 
with respect to either the frame $\bX^{j,q}_h$
or $\bX^q_h$, with $j\ge J_0$, $q\in \Xi$ and connecting 
$x$ and $y\in \B_\rE(x_0,\kappa_1R_0)$ is entirely contained in 
$\B_\rE(x_0,\kappa R_0)$.

 {\bf Step 4.} {\it Given $\kappa_1\in (0,\kappa)$ as in \eqref{eq:stimakappa1} and $\ep >0$, there exists $J_2=J_2(\ep,\kappa, \kappa_1)\in \N$ such that
\[
d_q^j(x,y)\leq d_q(x,y)+\ep,
\]
for every $j\geq J_2$, $x,y\in\B_\rE(x_0,\kappa_1R_0)$ and $q\in \Xi$.}

We fix $\kappa_1$ satisfying \eqref{eq:stimakappa1} and $\ep>0$, with  $x,y\in\B_\rE(x_0,\kappa_1R_0)$. 
By the previous step, we may consider
a SR geodesic $\gamma_{\bar h,x}^q\colon[0,d_q(x,y)]\to \B_\rE(x_0,\kappa R_0)$ 
connecting $x$ and $y$, for some $\bar h\in L^\infty([0,d_q(x,y)],\R^\m)$,
$\|\bar h\|_\infty=1$ and $q\in \Xi$.
We consider $\gamma_{\bar h,x}^{j,q}\colon[0,d_q(x,y)]\to \B_\rE(x_0,\kappa R_0)$ as in
\eqref{eq:Cauchykappa0}, with $h=\bar h$. 
These curves are all defined on the interval 
$[0,d_q(x,y)]$, due to \eqref{eq:Pedj} and \eqref{eq:Ckappa1} combined with Step~1.
We set $y_q^j=\gamma_{\bar h,x}^{j,q}(d_q(x,y))$ 
and apply \eqref{eq:stimejepsilon} for $j\ge J_2\coloneqq J_1((\ep/C_\kappa)^s,\kappa, \kappa_1)$, getting 
\[
|y_q^j-y|=|\gamma_{\bar h,x}^{j,q}(d_q(x,y))-\gamma_{\bar h,x}^q(d_q(x,y))|\le \pa{\frac{\ep}{C_\kappa}}^s.
\]
Due to \eqref{eq:C_Kdjq}, it follows that 
$d_q^j(y_q^j,y)\le \ep$, therefore
\beq\label{eq:dyjhx}
d_q^j(x,y)\le d_q^j(x,y_q^j)+d_q^j(y_q^j,y)\le d_q(x,y)+\ep
\eeq
for every $j\ge J_2$,
$x,y\in\B_\rE(x_0,\kappa_1R_0)$
and $q\in \Xi$. The inequality $d_q^j(x,y_q^j)\le d_q(x,y)$ is an obvious consequence of the definition of SR distance.

{\bf Step 5.} {\it Given $\kappa_1$ as in \eqref{eq:stimakappa1} and $\ep>0$ with 
$\Xi=\{p\}$, then there exists $J_3=J_3(\ep, \kappa, \kappa_1)\in\N$ and a subsequence $j_\ell$ such that for every $x,y\in \B_\rE(x_0,\kappa_1R_0)$, $j_\ell\geq J_3$ we have
\beq\label{ineq:d^q}
d_p(x,y)\leq d_p^{j_\ell}(x,y)+\varepsilon.
\eeq
}
Let us consider two arbitrary points $x,y\in\B_\rE(x_0,\kappa_1R_0)$. Step 3 of the proof shows that all geodesics 
$\gamma_{h^j,x}^{j,p}$ with respect to $\bX^{j,p}_h$ 
connecting $x$ and $y$ are contained in 
$\B_\rE(x_0,\kappa R_0)$.
Up to a rescaling, we may assume that 
$\gamma_{h^j,x}^{j,p}\colon[0,1]\to \B_\rE(x_0,\kappa R_0)$, $\gamma^{j,p}_{h^j,x}(1)=y$ and
\[
\|h^j\|_\infty=d_p^j(x,y)\le C_\kappa(2\kappa_1R_0)^{1/s}.
\]
By weak$^*$ compactness, there exists a subsequence of $h^j$ that converges to some $\bar h\in L^\infty([0,1],\R^\m)$
in the weak$^*$ topology. We may relabel $h^j$, so that  $h^j\stackrel{\!\!*}{\rightharpoonup}\bar h$. The uniform boundedness of $\gamma^{j,p}_{h^j,x}$ combined with the uniform convergence of $X^{j,p}_i$
to $X^p_i$, $i=1,\dots, \m$, implies their equi-Lipschitz continuity.
Thus, up to further relabeling $h^j$, by Ascoli--Arzelà compactness theorem, we also get the uniform convergence of $\gamma_{h^j,x}^{j,p}$ to some $\Gamma\colon[0,1]\to\B_\rE(x_0,\kappa R_0)$.
Actually, these conditions allow us to pass
to the limit in the following integral 
equation
\[
\gamma_{h^j,x}^{j,p}(t)=x+\int_0^t\sum_{i=1}^\m h^j_i(s)X^{j,p}_i(\gamma_{h^j,x}^{j,p}(s)) ds
\]
as $j\to\infty$, so that $\Gamma=\gamma_{\bar h,x}^p$ satisfies the second ODE of \eqref{eq:Cauchykappa0}. Moreover $\gamma_{\bar h,x}^p(1)=y$, therefore 
\[
d_p(x,y)\le \|\bar h\|_\infty \le\liminf_{j\to \infty} \|h^j\|_\infty
=\liminf_{j\to \infty} d_p^j(x,y)
\]
where $d_p^j(x,y)$ is also relabeled,
according to the subsequence of $h^j$.
Up to extracting a further subsequence,
we have actually proved the existence of $\tilde J_3=\tilde J_3(\ep,x,y, \kappa, \kappa_1)\in\N$ such that
\beq\label{eq:Jxy}
d_p(x,y)\le d_p^j(x,y)+\ep
\eeq
for all $j\ge \tilde J_3$.

By compactness, let $x_1,\dots, x_k\in \B_\rE (x_0,\kappa_1R_0)$ such that $\B_\rE(x_0,\kappa_1R_0)\subset \bigcup_{i=1}^kB_p(x_i,\varepsilon)$.
We use the notation $B_p$ and $B_p^j$ to denote the open metric balls with respect to $d_p$ and $d_p^j$, respectively. Thanks to \eqref{eq:C_Kdjq} and \eqref{eq:Jxy}  we can find $J_3=J_3(\ep, \kappa,\kappa_1,x_1,\ldots,x_k)$ such that 
\beq\label{eq:stimecentri}
\begin{aligned}
	& B_p(x_i,\ep)\subset B_p^j(x_i,C_\kappa^{1+1/s}\ep^{1/s})\quad &&i=1,\dots, k,\\
	& d_p(x_i,x_\ell)\le d_p^j(x_i,x_\ell) +\ep\quad && i,\ell=1,\dots, k,
\end{aligned}
\eeq
for every $j\ge J_3$. 
There exist $i_1,i_2$ such that $x\in B_(x_{i_1},\ep)$ and $y\in B_p(x_{i_2},\ep)$. Thanks to \eqref{eq:stimecentri}, by triangle inequality, for $j\geq J_3$ we have
\[
\begin{aligned}
d_p(x,y)&\le d_p(x,x_{i_1})+d_p(x_{i_1}, x_{i_2})+d_p(x_{i_2},y)\\
&\leq \ep +d_p^j(x_{i_1},x_{i_2})+2\ep\\
&\leq d_p^j(x_{i_1},x)+d_p^j(x,y)+d_p^j(y,x_{i_2})+3\ep\\
& \leq d_p^j(x,y)+2C_\kappa^{1+1/s}\ep^{1/s}+3\ep,
\end{aligned}
\]
which implies the desired inequality for the case $\Xi=\{p\}$. 
 
 {\bf Step 6.} {\it Theorem~\ref{th:DVGen} holds with $\Xi=\set{p}$.}
 
Combining Step 4, with $\Xi=\set{p}$, and Step 5, from the arbitrary choice of $\ep>0$, we obtain a subsequence $d_p^{j_\ell}$ such that 
\[
\lim_{\ell\to\infty}\max_{x,y\in \B_\rE(x_0,r_0)} |d_p^{j_\ell}(x,y)-d_p(x,y)|=0,
\]
with $r_0=\kappa_1R_0$. Since the limit is independent on the choice of the subsequence, our claim immediately follows.

{\bf Step 7.} {\it Given $\kappa_1$ as in \eqref{eq:stimakappa1} and $\ep>0$, there exists $J_3=J_3(\ep, \kappa, \kappa_1)\in\N$ such that
	\beq\label{ineq:d^q7}
	d_q(x,y)\leq d_q^j(x,y)+\varepsilon
	\eeq
	for every $x,y\in \B_\rE(x_0,\kappa_1R_0)$, $j\geq J_3$, $q\in \Xi$.
}

Assume by contradiction that there exist $\varepsilon_0>0$, some sequences $x_j, y_j$ in $\B_\rE(x_0,\kappa_1R_0)$ and $q_j$ in $\Xi$ such that
\begin{equation}\label{eq:assurdo}
	d_{q_j}(x_j,y_j)> d_{q_j}^{k_j}(x_j,y_j)+\varepsilon_0
\end{equation}
for every $j\in\N$, where $k_j\to\infty$.
By compactness, up to extracting subsequences, we can find $x_0,y_0 \in \B_\rE(x_0,\kappa_1R_0)$ and $q_0\in \Xi$ such that $x_j\to x_0$, $y_j\to y_0$ and $q_j\to q_0$ as $j\to \infty$, where we have relabeled the indices of the sequences and of $k_j$. 
Our assumptions on the convergence of the family $X_j^{q}$ imply that the sequences $X^{q_j}_1,\ldots,X^{q_j}_\m$ and $X^{k_j,q_j}_1,\ldots,X^{k_j,q_j}_\m$
both converge to $X_1^{q_0},\ldots,X^{q_0}_\m$ in the 
$C^\infty_{\rm loc}$-topology.
For the convergence of $X^{q_j}_1,\ldots,X^{q_j}_\m$, 
it is enough to use the continuity of all partial derivatives 
$(x,q)\mapsto \partial_x^\alpha X^q_i(x)$.
For the convergence of $X^{k_j,q_j}_1,\ldots,X^{k_j,q_j}_\m$,
we use the triangle inequality and the fact that $X_i^{j,q}\to X_i^q$ as $j\to \infty$ in the $C^\infty_{\rm loc}$-topology, uniformly with respect to $q\in U$, for any $i=1,\dots, \m$.

We apply Step 6 to both the sequences $\bX^{q_j}_h$ and $\bX^{k_j,q_j}_h$, hence the associated distances $d_{q_j}$ and $d_{q_j}^{k_j}$ uniformly converge to $d_{q_0}$ in $\B_\rE(x_0,\kappa_1R_0)\times\B_\rE(x_0,\kappa_1R_0)$. 
Thus, both $d_{q_j}(x_j,y_j)$ and $d_{q_j}^{k_j}(x_j,y_j)$ converge to $d_{q_0}(x_0,y_0)$ contradicting \eqref{eq:assurdo} and completing the proof of Step 7. Setting $r_0=\kappa_1R_0$, the proof is complete.
\end{proof}

\begin{rmk}
The previous theorem in some respects can be seen as a uniform version of \cite[Theorem~3.4]{DonVittone2019Comp} with respect to an ``extra parameter'' $q$ varying in an abstract compact metric space. Another difference is that we consider vector fields which are only defined on an open bounded set.
\end{rmk}

\begin{rmk}
	We notice that the possible values for $r_0>0$ in the statement of Theorem~\ref{th:DVGen} can be seen somehow more explicitly starting from the condition \eqref{eq:stimakappa1}. In fact, we obtain
	\[
	(\kappa_1R_0)^{1/s}+\kappa_1R_0\leq \frac{\kappa R_0}{4C_k\max\left\{C_\kappa 2^{1/s}, \frac 1{2C_\kappa}\right\}}=\frac{\kappa R_0}{4C_\kappa^2 2^{1/s} },
	\]
	taking into account the assumption $C_\kappa\geq 1$.
	Defining $\psi\colon [0,+\infty)\to [0,+\infty)$, $\psi(t)=t^{1/s}+t$, we can choose
	\[
	r_0=\psi^{-1}\left(\frac{\kappa R_0}{4C_\kappa^2 2^{1/s} }\right).
	\]
	The point is however that  $C_\kappa$ may depend on $R_0$, hence we de not know the behavior of the ratio in the previous expression as $R_0\to+\infty$. 
	The same formula could be useful only in those cases where more information on the behavior of $C_\kappa$ is available.
\end{rmk}


\begin{rmk}
Let us point out that the distances $d^j_q$ on $B_\rE(x_0,R_0)$ in the assumptions of Theorem~\ref{th:DVGen} may be a priori unbounded. Indeed, the norm of the vector fields of $\bX^{j,q}_h$ may tend to zero when it is evaluated at points that tend to the boundary. 
\end{rmk}

\begin{rmk}
To simplify the exposition, the previous theorem is stated with the
local $C^\infty$ convergence of the frames $\bX^{j,q}_h$. 
Since we may assume that the step of the frame is uniformly bounded,
a $C^k_{\rm loc}$ convergence with $k$ sufficiently large would suffice.
Such convergence is necessary to have the uniform estimates \eqref{eq:C_Kdjq}.
\end{rmk}

The main application of Theorem~\ref{th:DVGen} is the uniform nilpotent approximation of an equiregular sub-Riemannian manifold, as stated in the next theorem.

\begin{theorem}[Uniform nilpotent approximation]\label{thm:unifblowup}
Let $(M,\mathcal D, g)$ be a sub-Riemannian manifold and let $p\in M$ be a regular point.
We consider a privileged orthonormal frame $\bX=(X_1,\dots, X_\n)$ in an open neighborhood $W$ of $p\in M$ 
and denote by $F_\bX\colon U\times V\to W$ a system of uniform exponential coordinates of the first kind relative to $\bX$ with $p\in U$, 
see Definition~\ref{def:unifexpcoord} and \eqref{eqdef:exponentialcoordinates}. 
We set $\widetilde X_i^q= (F_\bX(q,\cdot)^{-1})_\ast X_i$
and introduce the rescaled vector fields $\widetilde X_i^{q,r}=r^{w_i}(\delta_{1/r})_*\widetilde X_i^q$ for $i=1,\ldots,\n$, $q\in U$ and $r>0$.
Then for every bounded open set $A\subset\R^\n$ the following statements hold.
\begin{enumerate}
	\item 
	The frame $\widetilde \bX^{q,r}=(\widetilde X^{q,r}_1,\ldots,\widetilde X^{q,r}_\n)$ converges to $\widehat \bX^q=(\widehat X^q_1,\ldots,\widehat X^q_\n)$ on the subset $A\subset\R^n$ in the $C^\infty_{\rm loc}$-topology as $r\to0$, uniformly with respect to $q$ varying in any compact set of $U$.
	\item 
	The local distance $\widetilde d_q^r$ induced by the frame $\widetilde \bX^{q,r}_h=(\widetilde X^{q,r}_1,\ldots,\widetilde X^{q,r}_\m)$ converges to $\widehat d_q$ in $L^\infty(A\times A)$ as $r\to 0$, uniformly as $q$ varies in any compact set of $U$,
	where $\widehat d_q\colon \R^n\times\R^\n\to[0,+\infty)$ is the tangent sub-Riemannian distance (Definition~\ref{def:nilpApprox}).
\end{enumerate}
\end{theorem}

\begin{proof}
We notice that the uniform coordinates $F_\bX$ exist, since $p$ is a regular point. For any $q\in U$, we set 
$$
\widetilde X_i^q= \sum_{j=1}^\n a_{ij}^q\,\partial_j
$$
for suitable smooth functions $a^q_{ij}\in C^\infty(V)$. We may apply Theorem~\ref{th:tangentspace} with respect to the
coordinates $F_\bX(q,\cdot)\colon V\to W$  for every $q\in U$. Then there exist 
smooth functions $\rho_{ij}^q\in C^\infty(V)$ and polynomials $b_{ij}^q\in C^\infty(V)$ such that 
$$
a_{ij}^q=b_{ij}^q+\rho_{ij}^q
$$
and the statements (a), (b) and (c) of Theorem~\ref{th:tangentspace} hold.
The homogeneity of the polynomials $b^q_{ij}$ implies that the vector fields $\widehat{X}^q_i=\sum_{j=1}^\n b^q_{ij}\partial_j$ are homogeneous, precisely
\[
(\delta_r)_*(\widehat{X}^q_i)=r^{w_i}\widehat{X}^q_i
\]
for every $i=1,\ldots,\n$ and $q\in U$.
As already observed in Remark~\ref{rmk:nilpotentapp}, the coefficients of the polynomial $b_{ij}^q$ are suitable iterated derivatives $a_{ij}^q$ at 0, with respect to our exponential coordinates of the first kind $x$ centered at $q$.
Since the uniform coordinates $F_\bX$ are smooth, then so are $(q,x)\mapsto a_{ij}^q(x)$, $(q,x)\mapsto b_{ij}^q(x)$ and $(q,x)\mapsto \rho_{ij}^q(x)$.
 
Let $K\subset U$ be compact and consider $r_1>0$ sufficiently small, such that such for
all $i=1,\ldots,\n$ and $0<r\le r_1$ the rescaled vector field $\widetilde X^{q,r}_i=r^{w_i}(\delta_{1/r})_*\widetilde X_i^q$ is well
defined on $A$ for every $q\in U$.
We wish to prove that $\widetilde X^{q,r}_i$ converges to $\widehat X_i^q$ in $B_\rE(0,R_0)$ with respect to the $C^\infty_{\rm loc}$-topology as $r\to 0$, uniformly with respect to $q\in K$ and for every $i=1,\ldots,\n$. 
The above homogeneity of $\widehat{X}^q_i$ implies that 
\[
r^{w_i}(\delta_{1/r})_*\widetilde X_i^q=\widehat X_i^q+\sum_{j=1}^\n r^{w_i-w_j}\rho_{ij}^q(\delta_r x)\partial_j
\]
for all $i=1,\ldots,\n$. The claim (c) of Theorem~\ref{th:tangentspace} states that the monomials in the Taylor expansion of $\rho_{ij}^q$ have homogeneous degree greater than $w_j-w_i$.
Then the smooth dependence of $\rho^q_{ij}(x)$ with respect to the couple $(q,x)$ implies that for any $\alpha\in \mathbb N^\n$, we have
\[
	\partial^{\alpha}_x\left[r^{w_i-w_j}\left(\rho_{ij}^q(\delta_r x)\right)\right]=O_{\alpha,i,j}(r),
\]
where $x\in A$ and $O_{\alpha,i,j}(r)$ is uniform with respect to $q\in K$. 
We have proved the first statement of our claim.

To prove the second statement, we fix again a compact set $K\subset U$. By the first claim with $A=\widehat B(0,4)$, we have that $r(\delta_{1/r})_*\widetilde X_i^q$ converges to $\widehat X_i^q$ on $\widehat B(0,4)$ with respect to the topology of $C^\infty_{\rm loc}$ as $r\to0$, for every $i=1,\ldots,\m$ and uniformly with respect to $q\in K$.
We denote by $\widehat d_{q,0}$ the distance induced by the orthonormal frame $\widehat \bX^q_h=(\widehat X_1^q,\ldots,\widehat X_\m^q)$ on $\widehat B_q(0,4)$, that may differ from the distance $\widehat d_q$ of the nilpotent approximation.
Let us recall that $\widehat B_q(0,4)$ is defined by $\widehat d_q$.
Obviously $\widehat d_{q,0}|_{\widehat B_q(0,1)}\ge {\widehat d_q}{|_{\widehat B_q(0,1)}}$,
hence we choose $x,y\in \widehat B_q(0,1)$.
Arguing as in the proof of Proposition~\ref{prop:localdist},
we may select a horizontal curve 
$\gamma\colon [0,1]\to \R^n$ connecting $x$ and $y$,
such that ${\rm length}(\gamma)<\widehat d_q(x,y)+\ep$ 
and $0<\ep<1$. We obtain
\[
\widehat d_q(\gamma(t),0)\leq \widehat d_q(\gamma(t),x)+\widehat d_q(x,0)< {\rm length}(\gamma)+1< 3+\varepsilon<4,
\]
therefore $\gamma(t)\in \widehat B_q(0,4)\subset $ for all $t\in[0,1]$.
In particular, we have 
\[
\widehat d_{q,0}(x,y)\le {\rm length}(\gamma)<\widehat d_q(x,y)+\ep,
\]
that immediately leads to the opposite inequality.
We have proved that 
\beq\label{eq:eqdistances&ind}
\widehat d_{q,0}|_{\widehat B_q(0,1)\times \widehat B_q(0,1)}=\widehat d_q|_{\widehat B_q(0,1)\times \widehat B_q(0,1)}.
\eeq
We apply Theorem~\ref{th:DVGen} to the converging frame $\widetilde \bX^{q,r}_h=(\widetilde X^{q,r}_1,\ldots,\widetilde X^{q,r}_\m)$
with some $0<R_0<1$ such that $B_\rE(0,R_0)\subset \widehat B_q(0,1)$, hence we get $0<r_0<R_0<1$ such that 
\beq\label{eq:unifr_0}
\sup_{q\in K}\sup_{x,y\in B_\rE(0,r_0)}|\widetilde d_q^r(x,y)-\widehat{d}_q(x,y)|=\sup_{q\in K}\sup_{x,y\in B_\rE(0,r_0)}|\widetilde d_q^r(x,y)-\widehat{d}_{q,0}(x,y)|\to0
\eeq
as $r\to0$, due to \eqref{eq:eqdistances&ind}. 
The previous limit and a standard rescaling argument conclude the proof.
\end{proof}

\section{Diameters of sub-Riemannian balls around regular points}\label{sec:diameters}

In this section, we prove local uniform estimates for the diameter of sub-Riemannian metric balls, which are centered at regular points. 

An important tool is the next theorem, which establishes the existence of a ``local uniform radius'' for all metric balls which are contained in the image of a suitable ``uniform topological exponential-type mapping''. 

\begin{theorem}[Topological existence of uniform radius]\label{thm:TopolUniformradius}
	Let $M$ be a length metric space and let $p\in M$. Let $T\subset M$ be an open neighborhood of $p$
	and let $A\subset \R^\n$ be an open neighborhood of $0$. 
	We consider a mapping 
	$E\colon T\times A\to M$ such that
	\begin{enumerate}
		\item\label{it1cont}
		$E$ is continuous,
		\item\label{it2Eq0}
		$E(q,0)=q$ for every $q\in T$,
		\item\label{it3hom}
		the mapping $E(q,\cdot)\colon A\to E(q,A)$ is a homeomorphism
		for every $q\in T$.
	\end{enumerate}
	Then there exist a bounded open neighborhood $V\subset A$ of $0$, an open neighborhood $U\subset T$ of $p$ such that the function
	\begin{equation}\label{eq:r_qpositive}
		U\ni q\mapsto R(q)=\sup\{t>0: B(q,t)\subset E(q,V) \}\in(0,+\infty)
	\end{equation}
	is well defined and lower semicontinuous.
	In particular, there exist $r_0>0$ and $\ep_0>0$ such that 
	$B(q,r_0)\subset E(q,V)$ for every $q\in B(p,\ep_0)$.
\end{theorem}
\begin{proof}
	We may choose two open neighborhoods $U\Subset T$ and $V\Subset A$ 
	of $p\in T$ and of $0\in A$, respectively, such that
	\begin{equation}\label{eq:InclEp}
		E(\overline U\times\overline V)\subset E(p,A).
	\end{equation}
	By conditions \eqref{it2Eq0} and \eqref{it3hom} and the fact
	that $V\Subset A$, we have $0<R(q)<+\infty$ for each $q\in U$,
	so that the function \eqref{eq:r_qpositive} is well defined. 
	We now fix a point $q\in U$ and $t>0$ such that $t<R(q)$
	and $B(q,2t)\subset U$.  
	By definition of $R(q)$, we may find $s\in(t,R(q))\cap(t,2t)$ such that
	\begin{equation}\label{eq:inclB(q,s)}
		\B(q,s)\subset U\cap E(q,V).
	\end{equation}
	From condition \eqref{it3hom}, we have 
	$\partial E(q,V)=E(q,\partial V)$, and hence 
	\[ 
	\B(q,s)\cap E(q,\partial V)=\emptyset.
	\]
	Then the triangle inequality implies
	\[
	B(q',t)\subset B(q,s),
	\]
	for every $q'\in B(q,\varepsilon)$ and $\ep\in(0,s-t)$. It remains to prove we can choose $\ep\in(0,s-t)$ sufficiently small such that
	\begin{equation}\label{eq:keyinclusion}
		B(q,s)\subset E(q',V).
	\end{equation}
	As a consequence of this inclusion, for each $q'\in B(q,\ep)$, we obtain
	\[
	B(q',t)\subset E(q',V).
	\]
	Our main tool will be the invariance of the degree of mappings
	under suitable deformations. 
	First of all, we claim that 
	\begin{equation}\label{eq:empty}
		\B(q,s)\cap E(q',\partial V)=\emptyset,
	\end{equation}
	for every $q'\in B(q,\varepsilon)$ and a fixed $\varepsilon>0$, sufficiently small.
	We can find $\sigma_0>0$ such that
	\[
	{\rm dist} (\B(q,s), E(q,\partial V))= {\rm dist }(\B(q,s), \partial E(q, V))=\sigma_0>0.
	\]
	By condition \eqref{it1cont} and the compactness of $\partial V$,
	for every $\delta\in (0,\sigma_0)$ we can find $\varepsilon>0$ such that
	\[
	\sup_{x'\in \partial V} d(E(q,x'),E(q',x'))<\delta<\sigma_0
	\] 
	for every $q'\in B(q,\varepsilon)$. For any $w\in \B(q,s)$ and $v\in \partial V$ we have
	\begin{equation*}
		d(w,E(q',v))\geq |d(w,E(q,v))-d(E(q,v),E(q',v))|\geq {\rm dist} (w, E(q,\partial V))-\delta \geq \sigma_0-\delta>0,
	\end{equation*}
	for every $q'\in B(q,\varepsilon)$. This concludes the proof of \eqref{eq:empty}. 
	Since $M$ is a length metric space,
	we can consider a continuous curve $\gamma\colon [0,1]\to B(q,\varepsilon)$ such that
	\[
	\gamma(0)=q\qquad \text{and}\qquad \gamma(1)=q'.
	\]
	Thus, in view of \eqref{eq:InclEp}, we can define the homotopy $H\colon \overline V\times[0,1]\to \R^\n$ by setting
	\[
	H(x,t)=E(p,\cdot)^{-1}\circ E(\gamma(t), x).
	\]
	By \eqref{eq:empty} we know that 
	\[
	E(\gamma(t),v)=E(p,H(v,t)) \notin B(q,s),
	\]
	for every $v\in \partial V$ and $t\in [0,1]$. 
	As a result, for an arbitrary
	$w\in E(p,\cdot)^{-1}(B(q,s))$ we have 
	\[
	w\neq H(v,t)
	\]
	whenever $v\in\partial V$ and $t\in[0,1]$.  
	Clearly the restriction $H(\cdot,0)$ is injective on $\overline V$.
	The inclusion $\B(q,s)\subset E(q,V)$, that is a consequence
	of \eqref{eq:inclB(q,s)} and the fact that 	
	$w\in E(p,\cdot)^{-1}(B(q,s))$ give $x\in V$ such
	that $w=E(p,\cdot)^{-1}(E(q,x))=H(x,0)$.
	By \cite[Theorem 3.3.3]{Lloyd78DegreeT}, we obtain that 
	\[
	{\rm deg}(H(\cdot,0), V, w)\in\{1,-1\}.
	\]
	We are then in a position to apply \cite[Thorem 2.3 (2)]{FonsecaGangbo95}, hence the function
	\[
	t\mapsto{\rm deg}(H(\cdot,t), V, w)
	\]
	is constant and in particular,
	\[
	{\rm deg}(H(\cdot,1), V, w)\in\set{1,-1}.
	\]
	By \cite[Theorem 2.1]{FonsecaGangbo95}, we have proved 
	that $E(p,w)\in E(q',V)$, hence 
	\[
	B(q,s)\subset E(q',V),
	\]
	that completes the proof.
\end{proof}

The previous result has an important consequence, when applied to our uniform exponential coordinates of the first kind, as explained in the next corollary.

\begin{corollary}[Existence of uniform radius]\label{cor:uniformradius}
	Let $(M,\mathcal D, g)$ be a sub-Riemannian manifold and let $p\in M$ be a regular point. Let $\bX=(X_1,\dots, X_\n)$ be a privileged orthonormal frame in an open neighborhood $W$ of $p\in M$ and consider a system of uniform exponential coordinates of the first kind $F_\bX\colon U\times V\to W$ with $p\in U$, 
	according to Definition~\ref{def:unifexpcoord} and \eqref{eqdef:exponentialcoordinates}. 
	Then there exist a bounded open neighborhood $V_0\subset V$ of $0$, an open neighborhood $U_0\subset U$ of $p$ and a positive $r_0>0$ such that $B(q,r_0)\subset F_\bX(q,V_0)$ for every $q\in U_0$.
\end{corollary}

\begin{proof}
We notice that $F_\bX\colon U\times V\to W$ satisfies the assumptions of Theorem~\ref{thm:TopolUniformradius}. Then there exists $r_0>0$ and $\ep_0>0$ such that 
\[
B(q,r_0)\subset F_\bX(q, V_0)
\]
for every $q\in B(p,\ep_0)$. This concludes the proof.
\end{proof}

\begin{theorem}[Uniform local estimates on diameters]\label{teo:diameter}
Let $p$ be a regular point of a sub-Riemannian manifold $(M,\mathcal D, g)$.
Then we have a neighborhood $T\subset M$ of $p$ such that for every $0<\ep<1$ there exists a radius $r_\ep>0$ such that
\begin{equation}
2r(1-\ep)	\le {\rm diam}(B(q,r))\le 2r
\end{equation}
for every $q\in T$ and $0<r<r_\ep$.
\end{theorem}
\begin{proof}
We choose a privileged orthonormal frame $\bX=(X_1,\dots, X_\n)$ in an open neighborhood $U$ of $p\in M$ and a system of uniform exponential coordinates of the first kind $F_\bX\colon U\times V\to W$  relative to $\bX$, according to Definition~\ref{def:unifexpcoord} and \eqref{eqdef:exponentialcoordinates}. 
Then we consider $\widetilde d^r_q$ and $\widehat d_q$ as in Theorem~\ref{thm:unifblowup}, therefore for every bounded
set $A\subset \R^\n$ and every compact neighborhood $K\subset U$ of $p$, we get
\begin{equation}\label{eq:dtildedcappuccio}
\lim_{r\to0}\;\sup\{|\widetilde d_q^r(x,y)-\widehat d_q(x,y)|: x,y\in A, q\in K\}=0.
\end{equation}
We consider $A=\widehat B_q(0,1)=\set{x\in\R^\n:\widehat d_q(x,0)<1}$, 
hence there exists $r_\ep>0$ such that 
the distances $\widetilde d^r_q$ are well defined on $\widehat B_q(0,1)$, $\widetilde d_q$ is well defined on $\widehat B_q(0,r)$  and we have
\[
-\ep<\frac{\widehat d_q(\delta_rx,\delta_ry)-\widetilde{d_q}(\delta_rx,\delta_ry)}{r}<\ep
\]
for all $x,y\in \widehat B_q(0,1)$, $q\in K$ and $0<r\le r_\ep$.
In different terms, we have
\[
-r\ep+\widetilde d_q(x,y)<\widehat d_q(x,y)<\widetilde d_q(x,y)+r\ep
\]
for each $x,y\in \widehat B_q(0,r)$, $q\in K$ and $0<r\le r_\ep$.
As a consequence, for $x\in \widehat B_q(0,r-r\ep)$ we obtain 
$\widetilde d_q(0,x)<r$, therefore
\[
\widehat B_q(0,r-r\ep)\subset \widetilde B_q(0,r)
\]
for every $q\in K$ and $0<r\le r_\ep$. Combining the inclusion and the previous inequality, for all $x,y\in \widehat B_q(0,r-r\ep)$, we get
\[
\widehat d_q(x,y)<\diam\widetilde B_q(0,r)+r\ep,
\]
that immediately yields
\beq\label{eq:unifestimradius}
2r-3r\ep \le \diam \widetilde B_q(0,r).
\eeq
From Corollary~\ref{cor:uniformradius} we get an open neighborhood $U_0\subset U$ of $p$ and $r_0>0$ such that 
\[
B(q,r_0)\subset F_\bX(q,V_0)
\]
for every $q\in U_0$, where $V_0\subset V$ is a bounded open neighborhood of $0$.
We fix an arbitrary compact neighborhood $T\subset U_0\subset U$ of $p$ and 
notice that the estimate \eqref{eq:unifestimradius} holds
for all $q\in T$ and $0<r\le r_\ep$, where $r_\ep$ may also depend on $T$.
Furthermore, for every $q\in T$ the diffeomorphism
$F_\bX(q,\cdot):V_0\to F_\bX(q,V_0)$ is a system of exponential coordinates of the first kind centered at the regular point $q$. We can apply Proposition~\ref{prop:localdist}, hence
\[
\widetilde{d}_q(x,y)=d(F_{q,\bX}(x),F_{q,\bX}(y))
\] 
holds for every $x,y\in F_{q,\bX}^{-1}(B(q,\rho))=\widetilde B_q(0,\rho)$ 
and every $q\in T$, where $0<\rho\le r_0/4$.
Due to \eqref{eq:unifestimradius}, we obtain
\[
2r-3r\ep \le \diam\widetilde B_q(0,r)=\diam B(q,r)
\]
for all $r\le \min\set{r_\ep,r_0/4}$ and $q\in T$, concluding the proof.
\end{proof}

\section{Finding the spherical Federer density of hypersurfaces}\label{sect:Federer}

The present section is devoted to the computation of the spherical Federer density of a hypersurface in an equiregular SR manifold at any of its noncharacteristic points. Terminology and tools will be also introduced in the next sections.

\subsection{Spherical Federer density and measure-theoretic area formula}\label{sect:FedDens}
In this section, we introduce the measure-theoretic area formula, which will lead us to 
the sub-Riemannian area formula \eqref{eq:SRareaperimeter}. We follow the presentation in \cite{LecMag21}.

\begin{definition}[Spherical measure and Hausdorff measure]\label{d:sizephi}\rm 
If $X$ is a metric space, $\delta>0$, $E\subset X$, $\cS$ is a fixed family of closed sets, $\set{F_j:j\in\N}\subset\cS$ covers $E\subset X$ and $\diam(F_j)\le \delta$ for every $j\in\N$, we say that this family is a {\em $\delta$-covering} of $E$. For $\alpha>0$, we set $\zeta_\alpha(F)= 2^{-\alpha}\, \diam(F)^\alpha$ for any $F\in\cS$ and introduce the set function
\beq
\phi_{\alpha,\delta}(E)=\inf \left\lbrace\sum_{j=0}^\infty \zeta_{\alpha}(E_j):
	\set{F_j:j\in\N}\subset\cS\;\text{is a $\delta$-covering of $E$} \right\rbrace\,.
\eeq
The Borel measure arising from this {\em Carathéodory's construction} is 
\beq
\phi^\alpha(E)=\sup_{\delta>0}\phi_{\alpha,\delta}(E).
\eeq
If $\cS$ is the family $\cF_b$ of all closed balls with positive diameter, then $\phi^\alpha$ becomes the {\em $\alpha$-dimensional spherical measure}, denoted by $\cS^\alpha$. 
If $\cS$ is the family of all closed sets, then $\phi^\alpha$ coincides with the $\alpha$-dimensional Hausdorff measure $\cH^\alpha$. 

If $k$ is a positive integer and $X$ is a finite dimensional Hilbert space,
we define the $k$-dimensional Euclidean Hausdorff measure as $\cH^k_\rE=\omega_k\cH^k$. Here the distance arises from the Euclidean distance and $\omega_k$ is the $k$-dimensional Lebesgue measure of the Euclidean unit ball in $\R^k$.
\end{definition}

The next definitions appear in the statement of the 
measure-theoretic area formula for the spherical measure (Theorem~\ref{th:sphericalarea}).

\begin{definition}[Federer density for the spherical measure] \label{def:FedDens}
We consider an outer measure $\mu$ over a metric space $X$ and fix $\alpha>0$.
The \emph{spherical Federer $\alpha$-density} of $\mu$ at $p$ is defined as
\begin{equation}\label{eq:Federerdensity}
		\cs^\alpha(\mu,p)= \inf_{\varepsilon >0}\sup\left\{\frac{2^\alpha\mu(\B)}{\diam ^\alpha(\B)}: p\in \B,\, 
		\B\in\cF_b, \,\diam (\B)<\varepsilon\right\}.
\end{equation}
\end{definition}

\begin{definition}[Diametric regularity]
A metric space $(X,d)$ is {\em diametrically regular}
if for every $x\in X$ there exist $\delta_x>0$ and $R_x>0$ such that $(0,+\infty)\ni r\to\diam(B(y,r))$ is continuous on $(0,\delta_x)$
for every $y\in B(x,R_x)$.
\end{definition}

The next result is the measure-theoretic area formula of \cite[Theorem~5.7]{LecMag21}.

\begin{theorem}[Measure-theoretic formula for the spherical measure]\label{th:sphericalarea}
Let $\mu$ be an outer measure over a metric space $X$ and let $\alpha>0$. 
We choose a Borel set $A\subset X$ and assume the validity of the following conditions.
	\begin{enumerate}
		\item
		$X$ is a diametrically regular metric space.		
		\item 
		$\mu$ is both a regular measure and a Borel measure.
		\item
		$(\cF_b)_{\mu,\zeta_{\alpha}}$ covers $A$ finely.
		\item
		$A$ has a countable covering whose elements are open and have $\mu$-finite measure.
		\item
		The subset $\lbrace x\in A: \cs^\alpha(\mu,x)=0\rbrace$ is
		$\sigma$-finite with respect to $\cS^\alpha$.
		\item
		We have the absolute continuity $\mu\res A<<\cS^\alpha\res A$.
	\end{enumerate}
	Then $\cs^\alpha(\mu,\cdot)\colon A\to[0,+\infty]$ is Borel and for every Borel set $B\subset A$ we have 
	\begin{equation}
		\mu(B)=\int_B\cs^\alpha(\mu,x)\,d\cS^\alpha(x).
	\end{equation}
\end{theorem}

\begin{rmk}\label{rem:metricarea}
It is not difficult to realize that Theorem~\ref{th:sphericalarea} combined with the double blow-up immediately gives \eqref{eq:SRareaperimeter}.
If we consider the general definition of Federer density, see \cite[Definition~3]{LecMag21}, then actually $(\cF_b)_{\mu,\zeta_{\alpha}}=\cF_b$, hence (3) is satisfied. The continuity of the diameter function in sub-Riemannian manifolds proves condition (1). 
Due to the finiteness of the Federer density \eqref{eq:theta=beta}, 
we can apply \cite[Proposition~3.3]{LecMag21}, hence getting the absolute continuity stated in (6). The remaining conditions are straightforward.
\end{rmk}

\subsection{Representation of the perimeter measure and double blow-up}

In this section we prove the integral formula \eqref{eq:perimetrorimaniano} for the perimeter measure, which will be used in Theorem~\ref{th:main} to compute the Federer density.


Using the standard properties of the interior product $i_X(\omega)$ of a vector field $X$ on the volume form $\omega$, we obtain the following variant of the divergence theorem.

\begin{theorem}\label{th:divergence}
	Let $(\Sigma,g)$ be an oriented Riemannian manifold of class $C^1$ with boundary $\partial \Sigma$ and let $X$ be a vector field on $\Sigma$ with compact support. Then, for every volume form $\omega$ on $M$, we have
	\[
	\int_\Sigma \dive_\omega X \,\omega=\int_{\partial \Sigma} g(X,\nu) \,\omega\res \nu,
	\]
	where $\nu$ denotes the outer normal to $\partial \Sigma$ with respect to $g$.
\end{theorem}

We recall that the divergence $\dive_\omega X$ with respect to $\omega$ is already introduced in \eqref{eq:divomega}.

\begin{rmk}\label{rmk:metric} In the special case where the volume form $\omega$ is the standard Riemannian volume form, the previous theorem yields the classical Divergence Theorem, see e.g.\ \cite[Theorem~1.47]{CLN2006Ricci}.
\end{rmk}

Theorem~\ref{th:divergence} is used in the proof of Theorem~\ref{th:formulaperimetro}, that also needs the next definitions.

\begin{definition}[Horizontal gradient]
Let $U$ be an open subset of $M$, where $(M,\mathcal D, g)$ is a sub-Riemannian manifold, and let $f\colon U\to \R$ be $C^1$ smooth. For every $p\in U$, we denote by $\nabla_\mathcal D f(p)$ the {\em horizontal gradient} of $f$ at $p$, namely the unique vector of $\mathcal D_p$ such that $\der f(p)(w)=g_p(\nabla_\cD f(p),w)$ for every $w\in\cD_p$.
\end{definition}

\begin{definition}[Horizontal normal and characteristic points]\label{def:hnor}
Let $(M,\mathcal D,g)$	be a sub-Riemannian measure manifold and let $\Sigma$ be an oriented $C^1$ smooth hypersurface $\Sigma$ embedded into $M$. 
Let $\bar g$ be a Riemannian metric on $M$ which extends the sub-Riemannian metric $g$ and let $\nu$ be a unit normal field on $\Sigma$. 
The {\em horizontal normal} $\nu_\mathcal D(p)$ at $p\in\Sigma$ is the orthogonal projection of $\nu(p)$ onto $\mathcal D_p$ with respect to $\bar g$.  We say that a point $p\in \Sigma$ is {\em characteristic} if $\cD_p\subset T_p\Sigma$, hence if and only if $\nu_\cD(p)=0$.
\end{definition}

\begin{definition}[Regular points for the divergence theorem] \label{def:regpbdry}
	Let $M$ be a manifold and let $\Omega\subset M$ be an open set. We say that $p\in \partial \Omega$ is a regular point of $\partial \Omega$ if there exist a neighborhood $U$ of $p$ and a map $f\in C^1(U)$ so that $\partial \Omega\cap U=\{q\in U: f(q)=0 \}$, $\dd f(q)\neq 0$ for all $q\in U$ and $\Omega \cap U=\{q\in U: f(q)<0\}$. 
\end{definition}

\begin{theorem}\label{th:formulaperimetro}
	Let $(M,\mathcal D,g,\omega)$ be an equiregular sub-Riemannian measure manifold and consider an open set $\Omega\subset M$ with $C^1$ smooth boundary. If $\bar g$ is any Riemannian metric on $M$ such that $\bar g|_{\mathcal D}=g$, then for every open set $U\subset M$, we have
	\begin{equation}\label{eq:perimetrorimaniano}
		\|D_{\omega,g}\bu_\Omega\|(U)=\int_{\partial \Omega \cap U} \|\omega\|_{\bar g}\, |\nu_{\mathcal D}|_{\bar g}\,\dd \sigma_{\bar g},
	\end{equation}
	where $\nu_\mathcal D$ denotes the projection on $\mathcal D$ with respect to $\bar g$ of the outer normal $\nu$ to $\partial \Omega$, and $\sigma_{\bar g}$ is the Riemannian surface measure associated with ${\bar g}$ of the boundary $\partial\Omega$.
\end{theorem}
\begin{proof}
	We first prove a local version of the statement. Let $p\in \partial \Omega\cap U$, let $V_p$ be an open and relatively compact neighborhood of $p$ that admits a positively oriented adapted frame $(e_1,\dots,e_\n)$ in $V_p$ with respect to $\omega$, 
	namely $\omega(e_1\wedge\cdots\wedge e_\n)>0$. We can assume that the frame is also orthonormal with respect to $\bar g$. Denote by $(\alpha_1,\dots,\alpha_\n)$ the dual basis of $(e_1,\dots,e_\n)$ and set the Riemannian volume form $\omega_{\bar g}=\alpha_1\wedge\ldots\wedge \alpha_\n$. Hence, we have a unique everywhere positive function $a\in C^{\infty}(V_p)$ such that $\omega= a\omega_{\bar g}$. 	
	By the divergence theorem stated in Theorem~\ref{th:divergence}, for any smooth vector field $X$ with compact support in $V_p$ it holds
	\begin{equation}\label{eq:divergenza}
		\int_{\Omega} \dive_{\omega} X \omega=\int_{\partial \Omega\cap V_p} a \,\bar g(X,\nu)\, \omega_{\bar g}\res\nu=\int_{\partial \Omega\cap V_p} a \,\bar g(X,\nu)\, \der \sigma_{\bar g},
	\end{equation}
	where the last equality follows taking into account the induced positive orientation on the boundary $\partial\Omega$ and Remark~\ref{rmk:metric}.
	Taking the supremum over all $\ph X$ with $X\in \cD^g$ and $\ph\in C_c^\infty(V_p)$ and $|\ph|\le 1$, we immediately obtain
	\[
	\|D_{\omega,g}\bu_\Omega\|(V_p)\leq\int_{\partial \Omega\cap V_p} a\, |\nu_\mathcal D|_{\bar g}\, \der\sigma_{\bar g},
	\]
	from Definition~\ref{def:SRPer}.
	For the opposite inequality, we first continuously extend $\nu_{\mathcal D}$ on $V_p$, using Tietze's Extension Theorem. Then we apply Stone--Weierstrass' theorem to get a sequence of smooth {\em horizontal vector fields} $\nu_k$ defined on $\overline V_p$ that uniformly converges to $\nu_\mathcal D$ on $V_p$. Indeed, it is enough to write 
	\[ 
	\nu=\sum_{i=1}^\m \nu^ie_i\quad\text{and}\quad 
	\nu_k= \sum_{i=1}^\m\nu_k^ie_i,
	\]
	where each $\nu_k^i$ uniformly converges to $\nu^i$ on $\overline V_p$ as $k\to \infty$.
	Take now $\phi\in C_c^\infty(V_p)$ such that 	$0\leq \phi \leq1$ and 
	define $X_k= \phi\frac{\nu_k}{|\nu_k|_g+e^{-k}}$. By evaluating identity \eqref{eq:divergenza} for $X=X_k$, we get
	\[
	\int_{\partial \Omega\cap V_p} a \,\bar g(X_k,\nu)\, \der\sigma_{\bar g}\leq \|D_{\omega,g}\bu_\Omega\|(V_p).
	\]  
	Passing to the limit as $k\to \infty$ we obtain
	\[
	\int_{\partial \Omega\cap V_p} \phi \,a\, |\nu_\mathcal D|_{\bar g}\, \der\sigma_{\bar g}\leq \|D_{\omega,g}\bu_\Omega\|(V_p),
	\]
	and passing to the supremum on all $\phi\in C_c^\infty(V_p)$ with $0\leq \phi\leq 1$, we get
	\[
	\|D_{\omega,g}\bu_\Omega\|(V_p)=\int_{\partial \Omega\cap V_p} a\, |\nu_\mathcal D|_{\bar g}\, \der\sigma_{\bar g}.
	\]
	Since $\|\omega_{\bar g}\|_{\bar g}=1$, we notice that  $a=\|\omega\|_{\bar g}$. 
	In sum, we have proved that for any $p\in \partial\Omega\cap U$ there exists an open and relatively compact set $V_p$ containing $p$, such that 
	\begin{equation}\label{eq:uguaglianzasullepallette}
	\|D_{\omega,g}\bu_\Omega\|(V_p)=\int_{\partial \Omega\cap V_p} \|\omega\|_{\bar g}\, |\nu_\mathcal D|_{\bar g}\, \der\sigma_{\bar g}.
	\end{equation}
	Clearly, the previous equality also holds on any open subset of $V_p$. 
	By paracompactness of $M$, we find a countable and locally finite open covering 
	\[
	\mathcal F= \{W_{i}:i\in \mathbb N\}
	\]
	of $\partial\Omega\cap U$ such that identity \eqref{eq:uguaglianzasullepallette} holds on each $W_i$. Let $\{\rho_i:i\in \mathbb N\}$ be a smooth partition of unity subordinate to $\mathcal F$ (see e.g.\ \cite[Theorem 13.7]{Loring2011}) and define the Borel regular measures
	\[
	\mu_1=\|D_{\omega,g}\bu_\Omega\|\quad\mbox{and}\quad \mu_2= \|\omega\|_{\bar g}\, |\nu_\mathcal D|_{\bar g}\, \der\sigma_{\bar g}\res\partial\Omega
	\]
	on $U$. Since $\sum_{i\in \mathbb N} \rho_i=1$, one has
	\[
	\mu_1=\sum_{i\in \mathbb N}\rho_i\mu_1\res W_i=\sum_{i\in \mathbb N} \rho_i \mu_2\res W_i=\mu_2,
	\]
	due to the fact that \eqref{eq:uguaglianzasullepallette} extends to all measurable subsets of $W_i$.
\end{proof}

Theorem~\ref{th:formulaperimetro} motivates the following definition of surface measure for hypersurfaces embedded in {\em sub-Riemannian measure manifolds}.

\begin{definition}[SR surface measure]\label{def:SRmeas}
We consider a sub-Riemannian measure manifold $(M,\mathcal D,g,\omega)$ along with an oriented $C^1$ smooth hypersurface $\Sigma$ of $M$. Let $\bar g$ be a Riemannian metric on $M$ such that $\bar g|_{\mathcal D}=g$ and let $\nu$ be an orienting unit vector field on $\Sigma$, that is orthogonal to $\Sigma$ with respect to $\bar g$. Then we define the {\em SR surface measure} of $\Sigma$ as
\begin{equation}\label{def:SRsurfmeas}
	\sigma^{SR}_\Sigma(A)=\int_{\Sigma\cap A} \|\omega\|_{\bar g}\, |\nu_{\mathcal D}|_{\bar g}\,\dd \sigma_{\bar g},
\end{equation}
for every Borel set $A\subset M$.
\end{definition}

\begin{theorem}[Double blow-up]\label{th:main} 
	Let $(M,\mathcal D, g, \omega)$ be a sub-Riemannian measure manifold. We assume that $\Sigma\subset M$ is an oriented $C^1$ smooth hypersurface with orienting unit normal $\nu$ on $\Sigma\cap T$ and $T\subset M$ is an open neighborhood of $p\in \Sigma$. 
	We also denote by the same symbol $g$ a Riemannian metric on $M$ that extends the given sub-Riemannian metric and we consider the associated SR surface measure $\sigma_\Sigma^{SR}$.
	If $p$ is both a regular point of $M$ and a noncharacteristic point of $\Sigma$, then
	\beq\label{eq:theta=beta}
	\cs^{Q-1}(\sigma_\Sigma^{SR}, p)=\|\omega(p)\|_g\, \beta_{d,g}(\nu_\cD(p)),
	\eeq
	where $\nu_\cD(p)$ denotes the horizontal normal at $p\in\Sigma$.
\end{theorem}
\begin{proof}
	Since $\Sigma$ is a hypersurface of class $C^1$, we may find a suitably small open neighborhood $ W$ of $p$ and $f\in C^1(W)$ such that 
	\[
	\Sigma\cap W=\{x\in  W: f(x)=0\},
	\]
	with $\dd f(q)\neq0$ and $q$ is a regular point of $M$ for every $q\in W$.
	Since $p$ is noncharacteristic, we have $\nabla_{\cD}f(p)\in\cD_p\sm\set{0}$, then we set
	\beq\label{eq:nu_dv_1}
	v_1= \frac{\nabla_\mathcal D f(p)}{|\nabla_\mathcal D f(p)|_g}=\theta_p\,\nu_\cD(p) \in \mathcal D_p
	\eeq
	for some $\theta_p\neq0$.
	Then we choose $v_2,\ldots,v_\m\in\cD_p$ such that $(v_1,\ldots,v_\m)$ is an orthonormal basis of $\cD_p$. Due to Lemma~\ref{lemma:estensioneadattata}, we can build a privileged orthonormal frame
	$(X_1,\dots, X_\n)$ such that 
	\[
	X_i(p)=v_i\quad \text{for all} \quad i=1,\ldots,\m.
	\]
	Up to changing the sign of one vector field $X_j$, with $1\le j\le \n$, we can also assume that the previous frame is positively oriented with respect to $\omega$, namely $\omega(X_1\wedge\cdots\wedge X_\n)>0$ on $W$. 
	By definition of $X_1(p)=v_1$, it follows that and
	\begin{equation}\label{eq:derivatenulle}
		X_1f(p)=|\nabla_\mathcal D f(p)|_g\quad\text{and}\quad X_2f(p)=\ldots=X_\m f(p)=0.
	\end{equation}
	We define $\bX=(X_1,\ldots,X_\n)$ and fix $F_{p,\bX}\colon V_p\to F_{p,\bX}(V_p)\subset W$, that introduces exponential coordinates of the first kind centered at $p$, where $F_{p,\bX}$ is a diffeomorphism and $V_p$ is an open neighborhood of $0\in\R^\n$, according Definition~\ref{def:unifexpcoord}. 
	
	To simplify notation we set $F=F_{p,\bX}$, since $p$ and $\bX$ are fixed throughout the proof. By the implicit function theorem applied to the composition $f\circ F$, 
	we find an open neighborhood $I$ of $0$ in $\R^{\n-1}$, an open interval $J\subset\R$ containing $0$ and a $C^1$ smooth function $\varphi\colon I\rightarrow J$ with $\varphi(0)=0$, such that 
	\begin{equation}\label{eq:implicit2}
		\Sigma\cap F(J\times I)=F\left(\{(\varphi(y),y): y\in I\}\right).
	\end{equation}
	We define $A=J\times I\subset V_p$ and consider the restriction $F\colon A \rightarrow U$, where $U=F(A)$. By our definitions 
	\begin{equation}\label{eq:G(t,X)}
		F(t,x)=\exp(t X_1+x_2X_2+\cdots+x_\n X_\n)(p)
	\end{equation} 
	for all $(t,x)\in A$, where $x=(x_2,\dots,x_\n)\in \R^{\n-1}$. 
	We also introduce the graph mapping $\Phi\colon I\to J\times I$ as $\Phi(x)=(\varphi(x), x)$ for every $x\in I$, so that
	\[
	\Sigma\cap F(A)=F\circ\Phi(I).
	\]
Combining the definition of $\sigma_\Sigma^{SR}$ from which
\beq
	\sigma_\Sigma^{SR}(\B(q,r))=\int_{\Sigma\cap \B(q,r)} \|\omega\|_g\,|\nu_\mathcal D|_g \, \dd\sigma_g,
\eeq
the change of variables \cite[Section 13.4.3, formula (12)]{BuragoZalgaller1988}, the Definition~\ref{def:balls} and the Proposition~\ref{prop:localdist}, we get
that $\sigma_\Sigma^{SR}(\B(q,r))$ is equal to 
\beq
\int_{\Phi^{-1}(\widetilde \B_p(F^{-1}(q),r))}\!\|\omega\|_g\!\circ\! F\!\circ\! \Phi\;\;|\nu_\mathcal D|_g\!\circ\! F\!\circ\! \Phi\; \sqrt{\det(g(\partial_i(F\circ\Phi),\partial_j(F\circ\Phi)))} \der \mathscr L^{\n-1},
\eeq
for every $q\in\B(p,r)\subset U$ and $r>0$ sufficiently small. 
	For $i,j=2,\dots, \n$, the notation $[g(\partial_i(F\circ\Phi),\partial_j(F\circ\Phi))]_{ij}$ indicates the corresponding $(\n-1)\times (\n-1)$ matrix. For $\tau>0$, define the dilation $\Lambda_\tau\colon \R^{\n-1}\rightarrow \R^{\n-1}$ by
	\[
	\Lambda_\tau(\xi_2,\dots,\xi_\n)=(\tau\xi_2,\dots, \tau\xi_\m,\tau^2\xi_{\m+1},\dots,\tau^2\xi_{\n_2},\dots, \tau^s\xi_\n).
	\]
	We also introduce the homogeneous norm 	$\|\xi\|=\sum_{i=1}^\n |\xi_i|^{1/w_i}$ for all $\xi\in \R^\n$. To simplify notation, we define 
	$\alpha(\tau,\eta)$ as the number given by the formula
	\[
	\|\omega\|_g(F(\Phi(\Lambda_\tau(\eta))))\,|\nu_\mathcal D|_g(F(\Phi(\Lambda_\tau(\eta)))) \sqrt{\det([g(\partial_i(F\circ\Phi),\partial_j(F\circ\Phi))]_{ij})(\Lambda_\tau(\eta))}
	\]
	for all $\tau>0$ and $\eta \in I$ such that $\Lambda_\tau(\eta)\in I$.
	For $r>0$ sufficiently small as above, we perform the change of variables $\eta\to\Lambda_r\eta$, obtaining
	\begin{equation}\label{eq:sigmadilatato2}
		\sigma_\Sigma^{SR}(\B(q,r))=r^{Q-1}\int_{\Lambda_{1/r}\Phi^{-1}(\widetilde \B_p(F^{-1}(q),r)))} \alpha(r,\eta) \,\der \mathscr L^{\n-1}(\eta).
	\end{equation}
	We first aim to prove that there exists $r_0>0$ and a compact set $K\subset\R^{\n-1}$ such that
	\begin{equation}\label{eq:equiboundedness2}
		\bigcup_{r\in (0,r_0]}\bigcup_{q\in \B(p,r)}\Lambda_{1/r}(\Phi^{-1}(\widetilde \B_p(F^{-1}(q),r)))\subset K.
	\end{equation}
	From Proposition~\ref{prop:localdist}, up to further shrinking
	both $A=J\times I$ and $U$, we get
	\begin{equation}\label{eq:cambiocoord2}
		\widetilde{d}_p(v_1,v_2)=d(F(v_1),F(v_2))
	\end{equation}
	for all	$v_1,v_2\in J\times I$. We observe that 
	\[
	w\in \Lambda_{1/r}(\Phi^{-1}(\widetilde \B_p(F^{-1}(q),r)))
	\]
	if and only if $\widetilde d_p(\Phi(\Lambda_r w),F^{-1}(q))\leq r$ with $0<r\le r_0$ and that $r_0>0$ can be chosen sufficiently small and independent of $q$. 
	Indeed, taking $r>0$ small enough, the sets 
	$\widetilde \B_p(F^{-1}(q),r)$ are all contained in a compact
	subset of $J\times I$, as $q$ varies in $\B(p,r)$. By \eqref{eq:cambiocoord2}, taking $0<r\le r_0$, 
	the previous inequality can be written as follows 
	\[ 
	d( F(\varphi(\Lambda_r w), \Lambda_r w), q)\leq r.
	\]
	This in turn implies that $d(F(\varphi(\Lambda_r w), \Lambda_r w), p)\leq 2r$, which equivalently writes 
	\begin{equation}\label{eq:boundedness2}
		\widetilde d_p((\varphi(\Lambda_r w), \Lambda_r w), 0)\leq 2r.
	\end{equation}
	By Theorem~\ref{th:localgroup}, observing that both $\widehat d_p$ and $\|\!\cdot\!\|$ are continuous and homogeneous with respect to dilations
		\[
		\delta_r(\xi_1,\dots,\xi_\n)=(r\xi_1,\dots,r\xi_\m,r^2\xi_{\m+1},\dots,r^2\xi_{\n_2},\dots,r^s\xi_\n),
		\]    
		up to shrinking $U, J$ and $I$, then there exists $C>1$ such that 
		\[
		\frac 1C \|\xi\|\leq \widetilde d_p(0,\xi)\leq C\|\xi\|
		\]
		for every $\xi \in J\times I$. Then inequality
	\eqref{eq:boundedness2} yields 
	\[
	\|(\varphi(\Lambda_r w), \Lambda_r w)\|\leq 2 Cr
	\] 
	for $0<r\le r_0$. This in turn implies that 
	\[
	\|w\| \leq \left|\frac{\varphi(\Lambda_r w)}{r}\right|+\|w\|\leq 2C
	\]
	and concludes the proof of \eqref{eq:equiboundedness2}. 	
	As a first consequence, we notice that 
\[
	\cs^{Q-1}(\sigma_\Sigma,p) <+\infty
\]
	and hence there exist a sequence $r_k>0$ converging to $0$ and a sequence $y_k\in \B(p,r_k)$ such that
	\[
	\cs^{Q-1}(\sigma^{SR}_\Sigma,p)=\lim_{k\to\infty} \frac{r_k^{Q-1}}{c_{Q-1}(\diam\,\B(y_k,r_k))^{Q-1}}\int_{A_k}\alpha(r_k,\eta)\;\der\mathscr L^{n-1}(\eta),
	\]
	where $A_k=\Lambda_{1/{r_k}}(\Phi^{-1}(F^{-1}(\B(y_k,r_k))))$. By Theorem~\ref{teo:diameter} and the choice $c_{Q-1}=2^{1-Q}$ we obtain 
	\begin{equation}\label{eq:densThetaFederer}
		\cs^{Q-1}(\sigma^{SR}_\Sigma,p)=\lim_{k\to \infty} \int_{A_k}\alpha(r_k,\eta)\;\der\mathscr L^{\n-1}(\eta).
	\end{equation}
	The dilations $\delta_r\colon\R\times \R^{\n-1}\to \R^\n$ can be
	written as $\delta_r(t,x)=(r t, \Lambda_r x)$. We consider
	an element $\delta_{1/{r_k}}F^{-1}(y_k)$, that is contained
	in $\delta_{1/{r_k}}\widetilde\B_p(0,r_k)$.
	
	Due to \eqref{incl:Btildeqreps}, for $\ep>0$ fixed and $k$ sufficiently large, we have
	\begin{equation}
		\widehat{d}_p(\delta_{1/{r_k}}F^{-1}(y_k),0)\le 1+\ep.
	\end{equation}
	Since $\ep>0$ is arbitrary, up to extracting a subsequence, we may assume that
	\begin{equation}
		\lim_{k\to\infty}\delta_{1/{r_k}}F^{-1}(y_k)=z\in \widehat \B_p(0,1).
	\end{equation}
	Next, we prove that for every $(0,w)\in (\{0\}\times \R^{\n-1}) \setminus \widehat\B_p(z,1)$ one has 
	\begin{equation}\label{eq:limitecaratteristiche1dinuovo}
		\lim_{k\to\infty}\chi_{A_k}(w)=0.
	\end{equation}
	Assume by contradiction that the previous limit does not hold.
	Up to selecting a subsequence, we can assume that $w\in A_k$ for every $k\in \mathbb N$. By definition of $A_k$, we have that 
	\[
	\Phi(\Lambda_{r_k}w)\in \widetilde\B_p(F^{-1}(y_k),r_k).
	\]
	In terms of distances, the previous condition becomes
	\begin{equation}\label{eq:dtildaerrecappa2}
		\widetilde d_p((\varphi(\Lambda_{r_k}w),\Lambda_{r_k}w),F^{-1}(y_k))\leq r_k
	\end{equation}
	for every $k\in \mathbb N$. Due to \eqref{eq:dtildaerrecappa2}, the triangle inequality allows us to apply Theorem~\ref{th:localgroup} with $R=2$, therefore 
	\begin{equation}
		\left|\frac{\widetilde d_p((\varphi(\Lambda_{r_k}w),\Lambda_{r_k}w),F^{-1}(y_k))}{r_k}-\widehat d_p\left(\left(\frac{\varphi(\Lambda_{r_k}w)}{r_k},w\right),\delta_{1/{r_k}}F^{-1}(y_k)\right)\right|<\varepsilon
	\end{equation}
	for every $k$ sufficiently large. Thus, taking into account \eqref{eq:dtildaerrecappa2}, we get
	\begin{equation}\label{eq:dcappuccio2}
		\widehat d_p\left(\left(\frac{\varphi(\Lambda_{r_k}w)}{r_k},w\right),\delta_{1/{r_k}}F^{-1}(y_k)\right)< 1+\varepsilon
	\end{equation}
	for $k$ sufficiently large.
	As a direct consequence of \eqref{eq:derivatenulle}, we have that
	\begin{equation}\label{eq:citti2}
		\lim_{r\to 0}\frac {\varphi{(\Lambda_r w)}}{r}=-\frac{\sum_{i=2}^\m w_i\, X_i f (p)}{X_1 f(p)}=0
	\end{equation}
	for any $w\in \R^{\n-1}$. Passing to
	the limit in $\eqref{eq:dcappuccio2}$, we get 
	\[
	\widehat d_p((0,w),z)\leq 1+\varepsilon.
	\]
	The arbitrary choice of $\ep>0$ gives $(0,w)\in \widehat\B_p(z,1)$, 
	that is a contradiction. This proves \eqref{eq:limitecaratteristiche1dinuovo}.
	We consider the set 
	\begin{equation}\label{eq:S_z}
		S_z=(\{0\}\times \R^{\n-1})\cap \widehat \B_p(z,1).
	\end{equation}
	Due to \eqref{eq:nu_dv_1}, we observe that the subspace 
	\beq\label{eq:Pinu}
	\Pi(\nu_\cD(p))=\dd F(0)(\{0\}\times \R^{\n-1})\subset T_pM
	\eeq
	is orthogonal to $\nu_\cD(p)$. Furthermore, $S_z$ can be seen as a subset of $\R^{\n-1}$, by means of the projection $\pi\colon\R\times\R^{\n-1}\to \R^{\n-1}$, hence it is
	easy to realize that 
	\begin{equation}\label{eq:S_zH^n-1}
		\mathscr L^{\n-1}(\pi(S_z))=\cH_E^{\n-1}(S_z),
	\end{equation}
	where $\cH^{\n-1}_E$ is the standard Hausdorff measure with respect to the Euclidean norm. We split the integral
	\[
	\int_{A_k}\alpha(r_k,\eta)\;\der\mathscr L^{\n-1}(\eta)=I_k+J_k,
	\]
	as follows  
	\[
	I_k=\int_{A_k\cap \pi(S_z)}\alpha(r_k,\eta)\;\der\mathscr L^{\n-1}(\eta)\quad\text{and}\quad J_k=\int_{A_k\setminus \pi(S_z)}\alpha(r_k,\eta)\;\der\mathscr L^{\n-1}(\eta).
	\]
	By \eqref{eq:equiboundedness2}, the sets $A_k$ are uniformly bounded in $K$, then the integrand of
	\[
	J_k= \int_{K\setminus \pi(S_z)}\chi_{A_k}(\eta)\alpha(r_k,\eta)\;\der\mathscr L^{\n-1}(\eta)
	\]
	is uniformly bounded by a summable function.
	By the dominated convergence theorem combined with \eqref{eq:limitecaratteristiche1dinuovo}, we obtain that
	$J_k\to0$ as $k\to \infty$, therefore \eqref{eq:densThetaFederer} gives
	\begin{equation}\label{eq:limitI_k}
		\lim_{k\to\infty}I_k=\cs^{Q-1}(\sigma_\Sigma^{SR},p).
	\end{equation}
	We notice that 
	\begin{equation}
		\alpha(r,\eta)\to \|\omega\|_g(p) |\nu_\mathcal D|_g (p) \sqrt{\det(g_p(\partial_i(F\circ\Phi),\partial_j(F\circ\Phi)))(0)}
	\end{equation}
	as $r\to 0$ uniformly on compact sets of $\R^{\n-1}$.
	We wish to simplify the expression of the previous limit.
	The vectors $e_2,\dots,e_\n$ denote the canonical basis of $\R^{\n-1}$. Then, for every $\eta\in I$ and every $i=2,\dots, \n$, one has 
	\begin{equation}
		\der(F\circ \Phi)(\eta)(e_i)= \der F(\Phi(\eta))(\partial_i\Phi(\eta))= \der F(\Phi(\eta))(\partial_i\varphi(\eta),e_i).
	\end{equation}
	We are interested in computing
	\[
	\det g_p\Big(\der F(0)\big(\partial_i\varphi(0)e_1+e_i\big), \der F(0)\big(\partial_j\varphi(0)e_1+e_j\big)\Big).
	\]
	Notice that, by definition of $F$ in \eqref{eq:G(t,X)}, $\der F(0)(e_i)=X_i(p)$, which have been chosen orthonormal with respect to $g_p$. As a consequence, 
	we get 
	\[
	g_p(\partial_i(F\circ\Phi)(0),\partial_j(F\circ\Phi)(0))= \partial_i\varphi(0)\partial_j\varphi(0)+\delta_{ij}
	\]
	for every $i,j=2,\dots, \n$. Notice that the entry $(i,j)$ of the matrix 
	$\der \Phi(0)^*\der \Phi(0)$ coincides with $\partial_i\varphi(0)\partial_j\varphi(0)+\delta_{ij}$, therefore 
	\[
	\det(\der \Phi(0)^*\der \Phi(0))= 1+|\nabla\varphi(0)|^2.
	\]
	By the implicit function theorem, one also gets
	\[
	\partial_j \varphi (0)=-\frac {\partial_j (f\circ F)(0)}{\partial_1 (f\circ F)(0)}
	\]
	for every $j=2,\dots,\n$. As consequence, considering also \eqref{eq:derivatenulle}, we can write 
	\[
	\begin{aligned}
		\sqrt{\det g_p(\partial_i(F\circ\Phi),\partial_j(F\circ\Phi))(0)}&= \sqrt{\det (\partial_i\varphi(0)\partial_j\varphi(0)+\delta_{ij})}\\&= \sqrt{1+|\nabla \varphi(0)|^2}\\&=\frac{|\nabla_g f|_g(p)}{|\nabla_\mathcal D f|_g(p)}=\frac{1}{|\nu_\mathcal D(p)|_g}.
	\end{aligned}
	\]
	We have proved that
	\begin{equation}\label{eq:omegap}
		\lim_{r\to0}\alpha(r,\eta)=\|\omega(p)\|_g
	\end{equation}
	uniformly on compact sets of $\R^{\n-1}$.
	Thus, combining \eqref{eq:limitI_k} and the inequality
	\[
	I_k\leq \int_{\pi(S_z)}\alpha(r_k,\eta)\;\der\mathscr L^{\n-1}(\eta),
	\]
	we obtain the upper estimate 
	\begin{equation}\label{ineq:thetaQ-1Ln-1}
		\cs^{Q-1}(\sigma_\Sigma^{SR},p)\leq \|\omega(p)\|_g\,\mathscr L^{\n-1}(\pi(S_{z_0})),
	\end{equation}
	where $z_0\in \widehat \B_p(0,1)$ is such that 
	$$
	\mathscr L^{\n-1}(\pi(S_{z_0}))=\max_{y\in \widehat \B_p(0,1)}\mathscr L^{\n-1}(\pi(S_y)).
	$$
	As a result, combining \eqref{eq:S_z}, \eqref{eq:Pinu}, \eqref{eq:S_zH^n-1} and Definition~\ref{def:beta}, we get
	\begin{equation}\label{eq:betaQ-1}
		\mathscr{L}^{\n-1}	(\pi(S_{z_0})) = \cH_E^{\n-1} (\dd F(0)^{-1}(\Pi(\nu))\cap\widehat \B_p(z_0,1))  =\beta_{d,g}(\nu_\cD(p)).
	\end{equation}
	It follows that
	\begin{equation}\label{ineq:thetaQ-1beta}
		\cs^{Q-1}(\sigma_\Sigma^{SR},p)\leq\|\omega(p)\|_g\beta_{d,g}(\nu_\cD(p)).
	\end{equation}
	To prove the opposite inequality of \eqref{ineq:thetaQ-1beta},  we consider the curve $y_0^r= F(\delta_r z_0)$ of parameter $r>0$
	and fix $\lambda >1$. For $\varepsilon_0>0$ arbitrarily fixed, by Theorem~\ref{thm:BallEquiv},
	there exists $r_0>0$ such that
	\begin{equation}
		z_0\in\delta_{1/r}(\widetilde{\B}_p(0,r(1+\ep_0)))
	\end{equation}
	whenever $0<r<r_0$. By definition of $y_0^r$, we have proved that
	$y_0^r\in \B(p,r(1+\varepsilon_0))$ for $0<r<r_0$.
	We arbitrarily choose $0<r'<r_0$.
	This easily implies that
	\begin{align*}
		\sup_{0<r<r'}\frac{\sigma_\Sigma^{SR}(\B(y_0^r,\lambda r))}{(\lambda r)^{Q-1}}&\leq (1+\varepsilon_0)^{Q-1} \sup_{0<r<\lambda r',\;y\in \B(p,r(1+\varepsilon_0))}\frac{\sigma_\Sigma^{SR}(\B(y, r(1+\varepsilon_0)))}{(1+\varepsilon_0)^{Q-1}r^{Q-1}}, \\
		&=(1+\varepsilon_0)^{Q-1} \sup_{0<r<\lambda r'(1+\ep_0),\;y\in \B(p,r)}\frac{\sigma_\Sigma^{SR}(\B(y, r))}{r^{Q-1}}.
	\end{align*}
	For arbitrary $\ep_1>0$, by Theorem~\ref{teo:diameter}, we can choose
	$r'>0$ sufficiently small such that 
	\begin{equation*}
		\sup_{0<r<r'}\frac{\sigma_\Sigma^{SR}(\B(y_0^r,\lambda r))}{(\lambda r)^{Q-1}}\leq(1+\varepsilon_0)^{Q-1}(1+\varepsilon_1)^{Q-1} 
		\sup_{0<r<\lambda r'(1+\ep_0),\;y\in \B(p,r)}\frac{2^{Q-1}\sigma_\Sigma^{SR}(\B(y, r))}{(\diam\B(y,r))^{Q-1}} 
	\end{equation*}
	The right-hand side of the previous inequality is less than or equal to 
	\begin{equation*}
		(1+\varepsilon_0)^{Q-1}(1+\varepsilon_1)^{Q-1}
		\sup_{p\in \B(y,r),\;\text{\scriptsize diam}\B(y,r)\le2\lambda r'(1+\ep_0)}\frac{2^{Q-1}\sigma^{SR}_\Sigma(\B(y, r))}{(\diam\B(y,r))^{Q-1}}.
	\end{equation*}
	Thus, from the definition of Federer density, taking the limit as $r'\to0$, it follows that	
	\begin{equation}
		\limsup_{r\to0}\frac{\sigma^{SR}_\Sigma(\B(y_0^r,\lambda r))}{(\lambda r)^{Q-1}}\leq (1+\varepsilon_0)^{Q-1}(1+\varepsilon_1)^{Q-1} \cs^{Q-1}(\sigma^{SR}_\Sigma,p).
	\end{equation}
	The arbitrary choice of $\varepsilon_0,\ep_1>0$ gives
	\begin{equation}\label{eq:limsup2}
		\limsup_{r\to0}\frac{\sigma^{SR}_\Sigma(\B(y_0^r,\lambda r))}{(\lambda r)^{Q-1}}\leq  \cs^{Q-1}(\sigma^{SR}_\Sigma,p).
	\end{equation}
	For $r>0$ sufficiently small, we set 
	\[
	A_0^r= \Lambda_{1/{(\lambda r)}}\Phi^{-1}(\widetilde \B_p(\delta_r z_0,\lambda r)),
	\]
	hence taking into account \eqref{eq:sigmadilatato2}, we get
	\beq\label{eq:densityratio}
	\frac{\sigma^{SR}_\Sigma(\B(y_0^r,\lambda r))}{(\lambda r)^{Q-1}}=\int_{A_0^r}\alpha(\lambda r, \eta)\;\der\mathscr L^{\n-1}(\eta)=\frac 1{\lambda^{Q-1}}\int_{\Lambda_\lambda A_0^r}\alpha(\lambda r,\Lambda_{1/\lambda}\eta)\;\der\mathscr L^{\n-1}(\eta).
	\eeq
	We wish to prove that
	\begin{equation}\label{eq:caratteristiche1dinuovo}
		\lim_{r\to0}\chi_{\Lambda_{\lambda}A_0^r}(w)=1,
	\end{equation}
	whenever $(0,w)\in (\{0\}\times\R^{\n-1})\cap \widehat \B(z_0,\lambda_1)$, with $1<\lambda_1<\lambda$. The extra parameter $\lambda_1$ allows us to use 
	the closed ball $\widehat \B_p(z_0,\lambda_1)$, since the definition \eqref{eq:S_z} of $S_z$ involves the closed ball. By contradiction, we assume that there exists an infinitesimal sequence $r_k>0$ such that $w\notin \Lambda_\lambda A_0^{r_k}$ for all $k\in \mathbb N$. By definition of $A_0^{r_k}$, we get
	\[
	\widetilde d_p(\Phi(\Lambda_{r_k}w),\delta_{r_k}z_0)>\lambda r_k.
	\]
	We know that $\delta_{r_k}z_0\in\widehat \B_p(0,r_k)$ and that the limit \eqref{eq:citti2} gives
	\beq\label{eq:lim0,w}
	\lim_{k\to\infty}\Lambda_{1/{r_k}}(\Phi(\Lambda_{r_k}w))=(0,w),
	\eeq
	therefore $\Phi(\Lambda_{r_k}w)\in\widehat \B_p(0,\lambda r_k)$ for any $k$ large enough. Then we may apply Theorem~\ref{th:localgroup} with $R=\lambda$ and for any fixed $1<\lambda'<\lambda$, we get 
	\[
	\widehat d_p\left(\left(\frac{\varphi(\Lambda_{r_k}w)}{r_k},w\right),z_0\right)>\lambda'
	\]
	for $k$ sufficiently large. By \eqref{eq:lim0,w} and the arbitrary choice of $\lambda'<\lambda$, we conclude that
	\[
	\widehat d_p((0,w),z_0)\geq \lambda>\lambda_1,
	\]
	that is a contradiction with the initial assumption $(0,w)\in (\{0\}\times \R^{\n-1})\cap \widehat \B_p(z_0,\lambda_1)$.
	We set 
		\[
		S_{z_0,\lambda}= (\{0\}\times \R^{\n-1})\cap \widehat\B_p(z_0,\lambda_1)
		\]
		and apply Fatou's lemma to \eqref{eq:densityratio}, so that
		combining \eqref{eq:limsup2}, \eqref{eq:caratteristiche1dinuovo} and \eqref{eq:omegap}, we get
		\[
		\lambda^{1-Q} \|\omega(p)\|_g\,\mathscr L^{\n-1}(\pi(S_{z_0,\lambda_1}))\leq \cs^{Q-1}(\sigma^{SR}_\Sigma,p).
		\]
	Letting first $\lambda_1\to \lambda^-$ and then $\lambda \to 1^+$, it follows that
	\[
	\|\omega(p)\|_g\mathscr L^{\n-1}(\pi(S_{z_0}))\leq \cs^{Q-1}(\sigma^{SR}_\Sigma,p).
	\]
	Due to \eqref{eq:betaQ-1}, the opposite inequality is also established, concluding the proof.
\end{proof}

\noindent 
{\bf Acknowledgements.} It is a pleasure to thank Davide Barilari, Antonio Lerario and Luca Rizzi for helpful comments and interesting discussions. We also thank Gioacchino Antonelli for useful comments.

\bibliographystyle{acm}
\bibliography{references}	
\end{document}